\documentclass[reqno]{amsart}

\usepackage[all]{xy}
\usepackage{graphicx}

\usepackage{amssymb}
\usepackage{amsmath}
\usepackage{mathrsfs}
\usepackage{epsfig}
\usepackage{amscd}
\usepackage{graphicx, color}

\def\E{\ifmmode{\mathbb E}\else{$\mathbb E$}\fi} 
\def\N{\ifmmode{\mathbb N}\else{$\mathbb N$}\fi} 
\def\R{\ifmmode{\mathbb R}\else{$\mathbb R$}\fi} 
\def\Q{\ifmmode{\mathbb Q}\else{$\mathbb Q$}\fi} 
\def\C{\ifmmode{\mathbb C}\else{$\mathbb C$}\fi} 
\def\H{\ifmmode{\mathbb H}\else{$\mathbb H$}\fi} 
\def\Z{\ifmmode{\mathbb Z}\else{$\mathbb Z$}\fi} 
\def\P{\ifmmode{\mathbb P}\else{$\mathbb P$}\fi} 
\def\T{\ifmmode{\mathbb T}\else{$\mathbb T$}\fi} 
\def\SS{\ifmmode{\mathbb S}\else{$\mathbb S$}\fi} 
\def\DD{\ifmmode{\mathbb D}\else{$\mathbb D$}\fi} 

\newcommand{\e}{\varepsilon}

\newcommand{\del}{\partial}

\newcommand{\ben}{\begin{enumerate}}
\newcommand{\een}{\end{enumerate}}
\newcommand{\be}{\begin{equation}}
\newcommand{\ee}{\end{equation}}
\newcommand{\bea}{\begin{eqnarray}}
\newcommand{\eea}{\end{eqnarray}}
\newcommand{\beastar}{\begin{eqnarray*}}
\newcommand{\eeastar}{\end{eqnarray*}}
\newcommand{\bc}{\begin{center}}
\newcommand{\ec}{\end{center}}

\theoremstyle{theorem}
\newtheorem{thm}{Theorem}[section]
\newtheorem{cor}[thm]{Corollary}
\newtheorem{lem}[thm]{Lemma}
\newtheorem{prop}[thm]{Proposition}

\theoremstyle{definition}
\newtheorem{defn}[thm]{Definition}
\newtheorem{rem}[thm]{Remark}
\newtheorem{ques}[thm]{Question}

\newtheorem{hypo}[thm]{Hypothesis}

\newtheorem*{thm*}{Theorem}

\numberwithin{equation}{section}

\def\R{{\mathbb R}}

\def\E{{\mathbb E}}
\def\Z{{\mathbb Z}}
\def\C{{\mathbb C}}
\def\R{{\mathbb R}}
\def\P{{\mathbb P}}

\def\N{{\mathbb N}}

\def\11{{\mathbb I}}

\def\delbar{{\overline \partial}}

\def\C{\mathbb{C}}
\def\Z{\mathbb{Z}}

\def\T{\mathbb{T}}

\def\Q{\mathbb{Q}}

\def\E{\ifmmode{\mathbb E}\else{$\mathbb E$}\fi} 
\def\N{\ifmmode{\mathbb N}\else{$\mathbb N$}\fi} 
\def\R{\ifmmode{\mathbb R}\else{$\mathbb R$}\fi} 
\def\Q{\ifmmode{\mathbb Q}\else{$\mathbb Q$}\fi} 
\def\C{\ifmmode{\mathbb C}\else{$\mathbb C$}\fi} 
\def\H{\ifmmode{\mathbb H}\else{$\mathbb H$}\fi} 
\def\Z{\ifmmode{\mathbb Z}\else{$\mathbb Z$}\fi} 
\def\P{\ifmmode{\mathbb P}\else{$\mathbb P$}\fi} 
\def\SS{\ifmmode{\mathbb S}\else{$\mathbb S$}\fi} 
\def\DD{\ifmmode{\mathbb D}\else{$\mathbb D$}\fi} 

\def\R{{\mathbb R}}

\def\E{{\mathbb E}}
\def\Z{{\mathbb Z}}
\def\C{{\mathbb C}}
\def\R{{\mathbb R}}

\def\N{{\mathbb N}}

\def\delbar{{\overline \partial}}

\def\e{\varepsilon}

\def\CH{{\mathcal H}}

\def\CJ{{\mathcal J}}

\def\CL{{\mathcal L}}

\def\CT{{\mathcal T}}

\def\darr#1{\raise1.5ex\hbox{$\leftrightarrow$}
\mkern-16.5mu #1}

\def\roughly#1{\raise.3ex\hbox{$#1$\kern-.75em
\lower1ex\hbox{$\sim$}}}

\def\opname#1{\mathop{\kern0pt{\rm #1}}\nolimits}

\def\Im{\opname{Im}}

\def\dist{\opname{dist}}

\def\pr{\prime}

\begin{document}

\quad \vskip1.375truein

\def\mq{\mathfrak{q}}
\def\mp{\mathfrak{p}}
\def\mH{\mathfrak{H}}
\def\mh{\mathfrak{h}}
\def\ma{\mathfrak{a}}
\def\ms{\mathfrak{s}}
\def\mm{\mathfrak{m}}
\def\mn{\mathfrak{n}}
\def\mz{\mathfrak{z}}
\def\mw{\mathfrak{w}}
\def\Hoch{{\tt Hoch}}
\def\mt{\mathfrak{t}}
\def\ml{\mathfrak{l}}
\def\mT{\mathfrak{T}}
\def\mL{\mathfrak{L}}
\def\mg{\mathfrak{g}}
\def\md{\mathfrak{d}}
\def\mr{\mathfrak{r}}

\def\hom{\textup{Hom}}
\def\aut{\textup{Aut}}
\def\diff{\textup{diff}}
\def\exp{\textup{exp}}
\def\id{\textup{id}}
\def\pr{\textup{pr}}
\def\doublearrow{\overrightarrow{\rightarrow}}
\def\bG{\bold{G}}
\def\bH{\bold{H}}
\def\bK{\bold{K}}
\def\bpsi{\bold{\Psi}}
\def\bphi{\bold{\Phi}}
\def\ric{\textup{Ric}}
\def\Int{\textup{Int}}
\def\loc{\textup{loc}}
\def\dist{\textup{dist}}

\title[Analysis of contact Cauchy--Riemann maps]{Analysis of Contact Cauchy--Riemann maps I: a priori $C^k$
estimates and asymptotic convergence}

\author{Yong-Geun Oh, Rui Wang}
\address{Center for Geometry and Physics, Institute for Basic Sciences (IBS), Pohang, Korea \&
Department of Mathematics, POSTECH, Pohang, Korea}
\email{yongoh1@postech.ac.kr}
\address{Department of Mathematics, University of California, Irvine, CA 92697}
\email{ruiw10@math.uci.edu}

\begin{abstract} In the present article, we develop tensorial analysis
for solutions $w$ of the following
nonlinear elliptic system
$$
\delbar^\pi w = 0, \, d(w^*\lambda \circ j) = 0,
$$
associated to a contact triad $(M,\lambda,J)$.
The novel aspect of this approach is that we work directly with this elliptic system on the contact manifold
without involving the symplectization process. In particular, when restricted to
the case where the one-form $w^*\lambda \circ j$ is exact, all a priori estimates for $w$-component can be written in terms
of the map $w$ itself without involving the coordinate from the symplectization.
We establish a priori $C^k$ coercive pointwise estimates for all $k \geq 2$
in terms of the energy density $\|dw\|^2$ by means of tensorial calculations on the
contact manifold itself. Further,  for any solution $w$ under the
finite $\pi$-energy assumption and the derivative bound, we also establish the asymptotic subsequence convergence to
`spiraling' instantons along the `rotating' Reeb orbit.
\end{abstract}


\thanks{This work is supported by IBS project \#IBS-R003-D1.}


\maketitle

\tableofcontents

\section{Introduction}
\label{sec:intro}

Let $(M,\xi)$ denote a $2n+1$ dimensional contact manifold $M$
equipped with contact structure $\xi$ (a completely non-integrable distribution of rank $2n$).
Moreover, assume that $\xi$ is co-oriented, so that one is able to
choose a one form $\lambda$ such that $\ker \lambda = \xi$. Such a one form is called a contact one form, and is not unique
but is determined only up to multiplication by nowhere vanishing functions.
Given a contact one form, the Reeb vector field $X_\lambda$ associated to it is uniquely
determined by the equations
$$
X_{\lambda} \rfloor \lambda \equiv 1, \quad X_{\lambda} \rfloor d\lambda \equiv 0.
$$
As an immediate consequence from the definition of contact structure, $(\xi, d\lambda|_\xi)$ is a symplectic vector bundle over $M$ of rank $2n$.
In the presence of the contact form $\lambda$,
one considers the set of endomorphisms $J: \xi \to \xi$ that are compatible with $d\lambda$ in
the sense that the bilinear form $g_\xi = d\lambda(\cdot, J\cdot)$ defines
a Hermitian vector bundle $(\xi,J, g_\xi)$ on $M$.
We call such an endomorphism $J$ a $CR$-almost complex structure.
As in \cite{blair}, we extend $J$ to an endomorphism of $TM$ by setting $JX_\lambda=0$.
We call the triple $(M, \lambda, J)$ a \emph{contact triad} and equip $M$ with the Riemannian metric
$$
g_\lambda =g_\xi +  \lambda\otimes\lambda
$$
which we refer to as the \emph{contact triad metric}. With the contact triad metric, a contact triad carries the
same information as a contact metric manifold. (See \cite{blair} and the references therein for more information about contact triads.)

Our goal is to understand the contact manifold without directly using its symplectization. Therefore,
we focus on maps $w:\dot\Sigma\to M$ from the (punctured)
Riemann surface $(\dot\Sigma, j)$ to the contact manifold $M$.
By decomposing the tangent bundle as $TM=\xi\oplus \R\{X_\lambda\}$ and denoting the projection to $\xi$ by $\pi$, one can further decompose $d^\pi w := \pi dw = \del^\pi w + \delbar^\pi w$ into the $J$-linear and anti-$J$-linear part as $w^*\xi$-valued
$1$-forms on the punctured Riemann surface $\dot\Sigma$.
We begin by considering maps $w$ satisfying just $\delbar^\pi w = 0$, which is
a nonlinear degenerate elliptic equation.

\begin{defn}[Contact Cauchy--Riemann Map] Let $(M,\lambda, J)$ be a contact triad
and let $(\dot\Sigma,j)$ be a (punctured) Riemann surface.
We call a smooth map $w: \dot\Sigma \to M$ a \emph{contact Cauchy--Riemann map}
if it satisfies $\delbar^\pi w = 0$.
\end{defn}

To maximize the advantage of using tensor calculus in the analytic study of
contact Cauchy--Riemann maps, we use the \emph{contact triad connection}
the authors introduced in \cite{oh-wang1} associated to the contact triad $(M,\lambda,J)$.
The contact triad connection, in particular, preserves the triad metric.
We review the contact triad connection in Section \ref{sec:connection}.

Denote by $\nabla$ the contact triad connection on $M$ and by $\nabla^\pi$ the associated
Hermitian connection on the Hermitian vector bundle $(\xi,d\lambda|_\xi,J)$.
Various symmetry properties of the connections $\nabla$ and $\nabla^\pi$
enable us to derive  precise formulae concerning the second covariant differential of
$w$ and the Laplacian of the $\pi$-harmonic energy density function
for any contact Cauchy--Riemann map $w$.

The following a priori on-shell equation for the $\pi$-harmonic energy density is
the basis of our a priori estimates for the contact Cauchy--Riemann map $w$.
This on-shell equation is the contact analog to the equation for symplectic manifolds derived
by the first-named author in Theorem 7.3.4 \cite{oh:book}.

\begin{thm}[Fundamental Equation]
Let $w$ be a contact Cauchy--Riemann map. Then
$$
d^{\nabla^\pi}(d^\pi w) = -w^*\lambda\circ j \wedge\left( \frac{1}{2}(\CL_{X_\lambda}J)\, d^\pi w\right).
$$
\end{thm}

The upshot of this equation is that the second derivatives (the left hand side) of $w$
are expressed in terms of the first derivatives of $w$ (the right hand side).

Define the $\xi$-component of the standard harmonic energy density function of general smooth map $w$ by
$$
e^\pi = e^\pi(w) : = |d^\pi w|_{g_\lambda}^2 :=  |\pi dw|_{g_\lambda}^2,
$$
and further introduce the following.
\begin{defn}\label{defn:pi-harmonic-energy} For any smooth map $w: \dot\Sigma \to M$,
the \emph{$\pi$-harmonic energy} $E_{(\lambda,J)}^\pi(w,j)$ of the smooth map $w$ is defined as
$$
E_{(\lambda,J)}^\pi(w,j): = \frac{1}{2} \int_{\dot \Sigma} e^\pi(w) =\frac{1}{2} \int_{\dot \Sigma} |d^\pi w|_{g_\lambda}^2.
$$
\end{defn}

Since we do not vary $\lambda$ or  $j$ or $J$ in the present article, we will use the shorthand notation $E^\pi(w)$ for $E_{(\lambda,J)}^\pi(w,j)$
from now on. Also, we will omit the subindex $g_\lambda$ from the norm $|\cdot|_{g_\lambda}$.

\begin{thm}
Let $w$ be a contact Cauchy--Riemann map. Then
\beastar
- \frac{1}{2} \Delta e^\pi & = & |\nabla^\pi(\del^\pi w)|^2 + K\, |\del^\pi w|^2
+ \langle  \mbox{\rm Ric}^{\nabla^\pi}(w) \del^\pi w, \del^\pi w\rangle\\
&{}&  + \langle  \delta^{\nabla^\pi} \left((w^*\lambda \circ j) \wedge (\CL_{X_\lambda}J) \del^\pi w\right), \del^\pi w \rangle
\eeastar
where $K$ is the Gaussian curvature of the given K\"ahler metric $h$ on $(\dot\Sigma,j)$
and $\ric^{\nabla^\pi}(w)$ is the Ricci curvature operator of the contact Hermitian connection
$\nabla^\pi$ along the map $w$.
\end{thm}

Again the upshot of this theorem is that for a contact Cauchy--Riemann map, the Laplacian of $e^\pi(w)$
which involves the 3rd derivatives of $w$ is expressible in terms of the second and the first derivatives of $w$.

Notice that due to the mismatch of the dimensions, the contact Cauchy--Riemann map equation itself is not an elliptic system.
To conduct geometric analysis, we
augment the equation $\delbar^\pi w = 0$ by an additional equation,
$$
d(w^*\lambda \circ j) = 0,
$$
and define the following.
\begin{defn}[Contact Instanton] Let $(\dot\Sigma, j)$ be a (punctured) Riemann surface and $w: \dot\Sigma \to M$ be a
smooth map. We call a pair $(j,w)$ a
\emph{contact instanton} if they satisfy
\be\label{eq:contact-instanton}
\delbar^\pi w = 0, \, \quad d(w^*\lambda \circ j) = 0.
\ee
\end{defn}

We would like to point out that the system \eqref{eq:contact-instanton} (for fixed $j$)
forms an elliptic system, which is a natural elliptic twisting of the degenerate
Cauchy--Riemann equation $\delbar^\pi w = 0$. (We refer to \cite{oh:sigmamodel} for an elaboration
of this point of view.)

Another  worthwhile  point  is that while the first  equation involves
first derivatives, the second  equation involves second derivatives of $w$.
Therefore it is not enough to have a $W^{2,2}$-bound to get a classical
solution out of a weak solution. Rather it is crucial to establish at least a $W^{3,2}$ coercive estimate
to start the standard bootstrapping arguments. With this in mind, we will derive an a priori local $C^k$-estimates
for contact instantons with the help of the contact triad connection.

We start with the following

\begin{thm}\label{thm:local-W12}
Let $(\dot \Sigma, j)$ be a punctured Riemann surface with a possibly
empty set of punctures. Equip $\dot \Sigma$ with a metric which is cylindrical
near each puncture. Let $w: \dot\Sigma \to M$ be a contact instanton.
For any relatively compact domains $D_1$ and $D_2$ in
$\dot\Sigma$ such that $\overline{D_1}\subset D_2$, we have
$$
\|dw\|^2_{W^{1,2}(D_1)}\leq C_1 \|dw\|^2_{L^2(D_2)} + C_2 \|dw\|^4_{L^4(D_2)},
$$
where $C_1, \ C_2$ are some constants which
depend only on $D_1$, $D_2$ and $(M,\lambda, J)$.
\end{thm}

We also establish the following iterative local $W^{2+k,2}$-estimates on punctured surfaces $\dot \Sigma$
in terms of the $W^{\ell,p}$-norms with $\ell \leq k+1$. Combined with
Theorem \ref{thm:local-W12}, this theorem in turn provides
a priori local $W^{2+k,2}$-estimates in terms of (local) $L^2$, $L^4$ norms of $|d^\pi w|$,
and $|w^*\lambda|$.

\begin{thm}\label{thm:local-regularity}
Let $w$ be a contact instanton.
Then for any pair of domains $D_1 \subset D_2 \subset \dot \Sigma$ such that $\overline{D_1}\subset D_2$,
$$
\int_{D_1} |(\nabla)^{k+1}(dw)|^2 \leq \int_{D_2} \CJ_{k}(d^\pi w, w^*\lambda).
$$
Here $\CJ_k$ is a polynomial function of degree up to $2k+4$ with nonnegative coefficients  of the norms of the covariant derivatives
of $d^\pi w, \, w^*\lambda$ up to $0, \, \ldots, k$ with degree at most $2k + 4$
whose coefficients depending on $J$, $\lambda$ and $D_1, \, D_2$ but independent of $w$.

In particular, any weak solution of \eqref{eq:contact-instanton} in
$W^{1,4}_{\loc}$ automatically becomes a classical solution.
\end{thm}

We refer to Theorem \ref{thm:local-higher-regularity}
and the discussions around them for further exposition on these estimates.

Next, we focus on cylindrical neighborhoods of the punctures and consider maps $w: [0,\infty) \times S^1\to M$
which satisfy
\eqref{eq:contact-instanton}.
There are natural asymptotic invariants $T$ and $Q$ which are defined as
\beastar
T & := & \frac{1}{2}\int_{[0,\infty) \times S^1} |d^\pi w|^2 + \int_{\{0\}\times S^1}(w|_{\{0\}\times S^1})^*\lambda\\
Q & : = & \int_{\{0\}\times S^1}((w|_{\{0\}\times S^1})^*\lambda\circ j).
\eeastar
Call $T$ the
\emph{asymptotic contact action} and $Q$  the
\emph{asymptotic contact charge}.

For the study of the asymptotic behavior of the contact instanton map near the punctures,
it is important to classify all possible \emph{massless instantons}
(i.e., instantons satisfying $E^\pi(w) = 0$) on the cylinder $\R \times S^1$ equipped with the standard complex
structure $j$. This classification of massless instantons
differs greatly between the $Q = 0$ and $Q \neq 0$ regimes.

\begin{prop}\label{prop:massless} Let $w: \R \times S^1 \to M$ be a massless contact instanton. Then
there exists a leaf of the Reeb foliation such that
we can write $w_\infty(\tau,t)= \gamma(-Q\, \tau + T\, t)$, where $\gamma$ is a parameterization
of the leaf satisfying $\dot \gamma = X_\lambda(\gamma)$.

In particular, if $T\neq 0$, $\gamma$ is a \emph{closed} Reeb orbit of $X_\lambda$ with period $T$.
In addition if $Q = 0$, $w_\infty$ is invariant under $\tau$-translations.

If $T = 0$ and $Q \neq 0$, the leaf needs  not be  closed but must be the image of an immersion of $\R$.
\end{prop}

With this classification result, we prove the following convergence result.
We refer readers to Theorem \ref{thm:subsequence}  for
more precise assumption for the following theorem.

\begin{thm}\label{thm:subsequence-intro} Let $w$ be any contact instanton on $[0,\infty) \times S^1$ with finite
$\pi$-harmonic energy
$$
E^\pi(w) = \frac{1}{2} \int_{[0,\infty) \times S^1} |d^\pi w|^2 < \infty,
$$
and finite gradient bound
$$
\|dw\|_{C^0;[0,\infty) \times S^1} < \infty.
$$
Then for any sequence $s_k\to \infty$, there exists a subsequence, still denoted by $s_k$, and a
massless instanton $w_\infty(\tau,t)$ (i.e., $E^\pi(w_\infty) = 0$)
on the cylinder $\R \times S^1$  such that
$$
\lim_{k\to \infty}w(s_k + \tau, t) = w_\infty(\tau,t)
$$
uniformly on $K \times S^1$ for any given compact set $K\subset\R$.

Furthermore if $Q = 0$ and $T \neq 0$, where $w_\infty(\tau, t) \equiv \gamma(T\, t)$ for some
closed Reeb orbit $\gamma$ of period $T$, the convergence is exponentially fast.
\end{thm}

Proposition \ref{prop:massless} and Theorem \ref{thm:subsequence-intro}  generalize Hofer's subsequence
convergence result in \cite{hofer}. Hofer's result in the context of symplectization, roughly corresponds to the exact case
(i.e., $Q = 0$ in our setting).
Our asymptotic analysis for the contact instanton equations reveals the new phenomenon of
`spiraling' instantons along a `rotating' Reeb orbit when the asymptotic charge is nonzero.
\medskip

As outlined above, our original motivation to study this new elliptic system lies in our attempt
to better understand the contact manifold itself instead of its symplectization.
Indeed the question of whether two contact manifolds having symplectomorphic symplectization are
contactomorphic or not was addressed in the book by Cieliebak and Eliashberg
(\cite[p.239]{ciel-eliash}).
Courte \cite{courte} provided a construction of two contact manifolds   which are not contactomorphic (actually, even not diffeomorphic) but have
symplectomorphic symplectizations.
It would be interesting to
see whether our approach could lead to a construction of genuinely contact
topological quantum invariants of the Gromov--Witten or Floer- theoretic type
that could be used to investigate the following kind of question.
(See \cite{courte} where a similar question was explicitly stated.)

\begin{ques}\label{ques:courte} Do there exist contact structures $\xi$ and $\xi'$ on a closed
manifold $M$ that have the same classical invariants and are not contactomorphic,
but whose symplectizations are
(exact) symplectomorphic?
\end{ques}

We would like also to recall a celebrated result by Ruan \cite{ruan} in symplectic geometry. Using Gromov--Witten invariants, he described a pair of algebraic surfaces which have the same classical invariants
but whose products with $S^2$ are not symplectically deformation equivalent.
\medskip

We note that a similar equation to \eqref{eq:contact-instanton} was first mentioned
by Hofer in p.698 of \cite{hofer-survey}. Then Abbas--Cielibak--Hofer in \cite{ACH} and Abbas \cite{abbas},
as well as by Bergmann in \cite{bergmann, bergmann2} used this equation to attack the Weinstein conjecture for dimension $3$. We would like to point out that their equations
correspond to our instanton equations of vanishing charge, i.e., $Q=0$.

However as far as the authors are aware of,  systematic a priori estimates without involving
symplectization such as those presented in this article have not been developed in
the previous literature. In this regard, our a priori estimates for $w$ are stronger than
those in the literature in that the $s$-coordinates do not enter in the a priori
estimates of $w$ even for the pseudoholomorphic maps $u = (s,w)$ in the context of symplectization. We hope that this kind of $s$-independent estimate for $w$
will lead to a better understanding of the convergence behavior of the contact instanton
$w$ even for the exact case. For this reason, we split Part I of the preprint
\cite{oh-wang2} and write the present article. We view this one self-contained
with focus on the tensorial derivation for a priori estimates.

\bigskip

\section{Review of the contact triad connection}
\label{sec:connection}

As defined in the introduction \ref{sec:intro}, assume $(M, \lambda, J)$ is a contact triad of dimension $2n+1$ for the contact manifold $(M, \xi)$, and equip with it the contact triad metric
$g=g_\xi+\lambda\otimes\lambda$.
In \cite{oh-wang1}, the authors introduced the \emph{contact triad connection} associated to every contact triad $(M, \lambda, J)$ with the contact triad metric and proved its existence and uniqueness.

\begin{thm}[Contact Triad Connection \cite{oh-wang1}]\label{thm:connection}
For every contact triad $(M,\lambda,J)$, there exists a unique affine connection $\nabla$, called the contact triad connection,
 satisfying the following properties:
\begin{enumerate}
\item The connection $\nabla$ is  metric with respect to the contact triad metric, i.e., $\nabla g=0$;
\item The torsion tensor $T$ of $\nabla$ satisfies $T(X_\lambda, \cdot)=0$;
\item The covariant derivatives satisfy $\nabla_{X_\lambda} X_\lambda = 0$, and $\nabla_Y X_\lambda\in \xi$ for any $Y\in \xi$;
\item The projection $\nabla^\pi := \pi \nabla|_\xi$ defines a Hermitian connection of the vector bundle
$\xi \to M$ with Hermitian structure $(d\lambda|_\xi, J)$;
\item The $\xi$-projection of the torsion $T$, denoted by $T^\pi: = \pi T$ satisfies the following property:
\be\label{eq:TJYYxi}
T^\pi(JY,Y) = 0
\ee
for all $Y$ tangent to $\xi$;
\item For $Y\in \xi$, we have the following
$$
\del^\nabla_Y X_\lambda:= \frac12(\nabla_Y X_\lambda- J\nabla_{JY} X_\lambda)=0.
$$
\end{enumerate}
\end{thm}
From this theorem, we see that the contact triad connection $\nabla$ canonically induces
a Hermitian connection $\nabla^\pi$ for the Hermitian vector bundle $(\xi, J, g_\xi)$, and we call it the \emph{contact Hermitian connection}. This connection will be used to study estimates for the $\pi$-energy in later sections.

The following remark provides some intuition of constructing the contact triad connection.
\begin{rem}
Recall that the leaf space of Reeb foliations of the contact triad $(M,\lambda,J)$
canonically carries a (non-Hausdorff) almost K\"ahler structure which we denote by
$
(\widehat M,\widehat{d\lambda},\widehat J).
$
We would like to note that Axioms (4) and (5)
are nothing but properties of the canonical connection
on the tangent bundle of the (non-Hausdorff) almost
K\"ahler manifold $(\widehat M, \widehat{d\lambda},\widehat J_\xi)$
lifted to $\xi$. In fact, as in the almost K\"ahler case, vanishing of the $(1,1)$-component
also implies vanishing of the $(2,0)$-component and hence the torsion must be of  $(0,2)$-type.
On the other hand, Axioms (1), (2) and (3)
indicate this connection behaves like the Levi-Civita connection along the Reeb direction $X_\lambda$.
Axiom (6) is an extra requirement to connect the information in the $\xi$ part and the $X_\lambda$ part,
which uniquely pins down the desired connection.
\end{rem}

Moreover, the following fundamental properties of the contact triad connection was
proved in \cite{oh-wang1}, which will be used to perform tensorial calculations later.

\begin{cor}\label{cor:connection}
Let $\nabla$ be the contact triad connection. Then
\begin{enumerate}
\item For any vector field $Y$ on $M$,
\be\label{eq:nablaYX}
\nabla_Y X_\lambda = \frac{1}{2}(\CL_{X_\lambda}J)JY;
\ee
\item $\lambda(T|_\xi)=d\lambda$.
\end{enumerate}
\end{cor}

We refer readers to \cite{oh-wang1} for more discussion on the contact triad connection and its relation with other related canonical type connections. In particular, we would like to remark that
\begin{rem}Using the identity
$\CL_{X_\lambda}\lambda=0=\CL_{X_\lambda}d\lambda$ it is not hard to see that,
the Reeb vector field is a Killing vector field with respect to the triad metric if and only if $\CL_{X_\lambda}J=0$.
In general, this is a strong additional requirement. For example, for $3$-dimensional contact manifolds, it is equivalent to the Sasakian condition.
Hence, for the contact triad connection, $\nabla X_\lambda$ doesn't vanish, which indicates that it is different from the canonical connection for the symplectization when lifted. For more details regarding this, refer \cite{oh-wang1} and the references therein.
\end{rem}

This section ends with introducing the following notation for later use.
Associated to the projection $\pi=\pi_\lambda$ from $TM$ to $\xi$,
we use $\Pi=\Pi_\lambda: TM\to TM$ to denote the corresponding idempotent, i.e., the endomorphism of $TM$ satisfying
${\Pi}^2 = \Pi$, $\Im \Pi = \xi$, $\ker \Pi = \R\{X_\lambda\}$.

\section{The contact Cauchy--Riemann maps}
\label{sec:CRmap}

Denote by $(\dot\Sigma, j)$ a punctured Riemann surface (including the case of closed Riemann surfaces without punctures).

\begin{defn}A smooth map $w:\dot\Sigma\to M$ is called a \emph{contact Cauchy--Riemann map}
(with respect to the contact triad $(M, \lambda, J)$), if $w$ satisfies the following Cauchy--Riemann equation
$$
\delbar^\pi w:=\delbar^{\pi}_{j,J}w:=\frac{1}{2}(\pi dw+J\pi dw\circ j)=0.
$$
\end{defn}

Recall that for a fixed smooth map $w:\dot\Sigma\to M$,
the triple $(w^*\xi, w^*J, w^*g_\xi)$ becomes  a Hermitian vector bundle
over the punctured Riemann surface $\dot\Sigma$.  This  introduces a Hermitian bundle structure on
$Hom(T\dot\Sigma, w^*\xi)\cong T^*\dot\Sigma\otimes w^*\xi$ over $\dot\Sigma$,
with inner product given by
$$
\langle \alpha\otimes \zeta, \beta\otimes\eta \rangle =h(\alpha,\beta)g_\xi(\zeta, \eta),
$$
where $\alpha, \beta\in\Omega^1(\dot\Sigma)$, $\zeta, \eta\in \Gamma(w^*\xi)$,
and $h$ is the K\"ahler metric on the punctured Riemann surface $(\dot\Sigma, j)$.

Let $\nabla^\pi$ be the contact Hermitian connection.
Combining the pulling-back of this connection and the Levi-Civita connection of the Riemann surface,
we get a Hermitian connection for the bundle $T^*\dot\Sigma\otimes w^*\xi \to \dot\Sigma$.
By a slight abuse of notation, we will still denote by $\nabla^\pi$ this combined connection.

The smooth map $w$ has an associated $\pi$-harmonic energy density
defined as the norm of the section $d^\pi w:=\pi dw$ of $T^*\dot\Sigma\otimes w^*\xi\to \dot\Sigma$.
In other words, it is the function $e^\pi(w):\dot\Sigma\to \R$ defined by
$
e^\pi(w)(z):=|d^\pi w|^2(z).
$
(Here we use $|\cdot|$ to denote the norm from $\langle\cdot, \cdot \rangle$ which should be clear from the context.)

Similarly to the case of pseudoholomorphic curves on almost K\"ahler manifolds,
we obtain the following basic identities.

\begin{lem}\label{lem:energy-omegaarea}
Fix a K\"ahler metric $h$ on $(\dot\Sigma,j)$,
and consider a smooth  map $w:\dot\Sigma \to M$.  Then we have the following equations
\begin{enumerate}
\item $e^\pi(w):=|d^\pi w|^2 = |\del^\pi w| ^2 + |\delbar^\pi w|^2$;
\item $2\, w^*d\lambda = (-|\delbar^\pi w|^2 + |\del^\pi w|^2) \,dA $
where $dA$ is the area form of the metric $h$ on $\dot\Sigma$;
\item $w^*\lambda \wedge w^*\lambda \circ j = - |w^*\lambda|^2\, dA$.
\end{enumerate}
As a consequence, if $w$ satisfies $\delbar^\pi w=0$, then
\be\label{eq:onshell}
|d^\pi w|^2 = |\del^\pi w| ^2 \quad \text{and}\quad w^*d\lambda = \frac{1}2|d^\pi w|^2 \,dA.
\ee
\end{lem}
\begin{proof} The proofs of (1) and (2) are exactly the same as the case of pseudo-holomorphic maps
in symplectic manifolds replacing $dw$ by $d^\pi w$ and the symplectic
form by $d\lambda$ and so they are omitted. See e.g., Proposition 7.2.3 \cite{oh:book} for the
statements and their proofs in the symplectic case corresponding the statements (1) and (2) here.
Statement (3) follows from the
definition of the Hodge star operator which implies that for any $1$-form $\beta$ on the Riemann surface
$*\beta=-\beta\circ j$, and we take $\beta=w^*\lambda$.
\end{proof}

Notice that the contact Cauchy--Riemann equation itself is \emph{not} an elliptic system
since the symbol is of rank $2n$ which is $1$ dimension lower than $TM$.
Here the closedness condition $d(w^*\lambda\circ j)=0$
leads to an elliptic system (see \cite{oh:sigmamodel} for an explanation)
\be\label{eq:instanton}
\delbar^\pi w=0, \quad d(w^*\lambda\circ j)=0.
\ee
\begin{defn} We call a solution of the system \eqref{eq:instanton} a \emph{contact instanton}
\end{defn}
Contact instantons are the main objects of our study in the present paper.

To illustrate the effect of the closedness condition on the behavior of
contact instantons, we look at them on \emph{closed}
Riemann surface and prove the following classification
result. The following proposition is stated by Abbas as a part of \cite[Proposition 1.4]{abbas}. For readers' convenience,
we separate this part for closed contact instantons (which are called  homologically
perturbed pseudo-holomorphic curves in \cite{abbas}) and give
a somewhat different proof.

\begin{prop}\label{prop:abbas} Assume $w:\Sigma\to M$ is a smooth contact instanton from a closed Riemann surface.
Then
\begin{enumerate}
\item If $g(\Sigma)=0$, $w$ is a constant map;
\item If $g(\Sigma)\geq 1$, $w$ is either a constant or the locus of its image
is a \emph{closed} Reeb orbit.
\end{enumerate}
\end{prop}
\begin{proof}For contact Cauchy--Riemann maps, Lemma \ref{lem:energy-omegaarea} implies
that $|d^\pi w|^2\, dA =d(2w^*\lambda)$.
By Stokes' formula, we get $d^\pi w=0$ if the domain is a closed Riemann surface,
and further, $dw=w^*\lambda \otimes X_\lambda$, i.e., $w$ must have its image contained in
a single leaf of the (smooth) Reeb foliation.

Another consequence of the vanishing $d^\pi w=0$
is that $dw^*\lambda = 0$.
Now this combined with the equation $d(w^*\lambda\circ j)=0$, which is equivalent to $\delta w^*\lambda=0$,
implies that $w^*\lambda$ (so is $*w^*\lambda$) is a harmonic $1$-form on the Riemann surface $\Sigma$.

If the genus of $\Sigma$ is zero, $w^*\lambda = 0$ by the Hodge's theorem.
This proves statement (1).

Now assume $g(\Sigma) \geq 1$. Suppose $w$ is not a constant map. Since $\Sigma$ is compact and connected,
$w(\Sigma)$ is compact and connected. Furthermore recall $w(\Sigma)$ is contained in a single leaf
of the Reeb foliation which we denote by
$\CL$. We take a parametrization $\gamma: \R \to \CL \subset M$ such that $\dot \gamma = X_\lambda(\gamma(t))$.
By the classification of compact one dimensional manifolds, the image $w(\Sigma)$ is homeomorphic
either to the unit closed interval or to the circle. For the latter case, we are done.

For the former case, we let $I$ denote $\omega(\Sigma)$ which is contained in the leaf $\CL$.
We slightly extend the interval $I$ to $I' \subset \CL$ so that $I'$ still becomes an
embedded interval contained in $\CL$. The preimage $\gamma^{-1}(I')$ is a disjoint union of
a sequence of intervals $[\tau_k, \tau_{k+1}]$ with $\cdots < \tau_{-1} < \tau_0 < \tau_1 < \cdots$ for
$k \in \Z$. Fix any single interval, say,
$[\tau_0, \tau_1] \subset \R$.

We denote by $\gamma^{-1}: I' \to [\tau_0,\tau_1] \subset \R$ the inverse of
the parametrisation $\gamma$ restricted to $[\tau_0,\tau_1]$. Then by construction
$\gamma^{-1}(I) \cap [\tau_0,\tau_1] \subset (\tau_0,\tau_1)$.

Now we denote by $t$ the standard coordinate function of $\R$ and
consider the composition $f: = \gamma^{-1} \circ w: \Sigma \to \R$. It follows that $f$ defines
a smooth function on $\Sigma$ satisfying
$$
\gamma \circ f = w
$$
on $\Sigma$ by construction. Then recalling $\dot \gamma = X_\lambda(\gamma)$, we obtain
$$
w^*\lambda = f^*(\gamma^*\lambda) = f^*(dt) = df.
$$
Therefore $\Delta f = \delta df = \delta w^*\lambda = 0$, i.e., $f$
is a harmonic function on the closed surface $\Sigma$ and so must be a constant function.
This in turn implies $w^*\lambda = 0$. Then $dw = d^\pi w + w^*\lambda X_\lambda(w) = 0 + 0 = 0$
i.e., $w$ is a constant map which contradicts the standing hypothesis. Therefore
the map $w$ must be constant unless the image of $w$
wraps up a closed Reeb orbit.

\end{proof}

\section{Calculation of the Laplacian of $\pi$-harmonic energy density}
\label{sec:calculations}

In this section, we use the contact triad connection to derive some identities related to the $\pi$-harmonic energy for
contact Cauchy--Riemann maps.
Our derivation is based on \emph{coordinate-free} tensorial calculations.
The contact triad connection fits well for this purpose which will be seen clearly in this section.

We start with looking at the (Hodge) Laplacian of the $\pi$-harmonic energy density
of an arbitrary smooth map $w: \dot \Sigma \to M$, which is not necessarily contact Cauchy--Riemann,
 i.e., in the \emph{off-shell level} in physics terminology.
As the first step, we apply the standard Weitzenb\"ock formula to the connection $\nabla^\pi$
on $T^*\dot\Sigma \otimes w^*\xi$ that is induced by the the pull-back connection on
bundle $w^*\xi$ and the Levi-Civita connection on $T\dot \Sigma$, and obtain
the following formula
\bea
-\frac{1}{2}\Delta e^\pi(w)&=&|\nabla^\pi(d^\pi w) |^2-\langle \Delta^{\nabla^\pi} d^\pi w, d^\pi w\rangle
+K\cdot |d^\pi w|^2+\langle \ric^{\nabla^\pi}(d^\pi w), d^\pi w\rangle.\nonumber\\
&&\label{eq:bochner-weitzenbock-e}
\eea
Here $e^\pi:=e^\pi(w)$, $K$ is the Gaussian curvature of $\dot\Sigma$,
and $\ric^{\nabla^\pi}$ is the Ricci tensor of the connection $\nabla^\pi$
on the vector bundle $w^*\xi$.
(For readers' convenience, we give the proof of  \eqref{eq:bochner-weitzenbock-e} in Appendix \ref{appen:weitzenbock}.
For the basic differential notations, such as $d^\nabla$, $\delta^\nabla$ etc.,
we also refer readers to that section if necessary.)

Next we derive an important expression for $d^{\nabla^\pi}d^\pi w$ in the off-shell level, which
is the analog to a similar formula
\cite[Lemma 7.3.2]{oh:book} in the symplectic context.

\begin{lem}\label{lem:FE-autono} Let $w: \dot\Sigma \to M$ be any smooth map. Denote by $T^\pi$ the torsion tensor of $\nabla^\pi$.
Then as a two form with values in $w^*\xi$,
$d^{\nabla^\pi} (d^\pi w)$ has the expression
\be\label{eq:dnabladpiw}
d^{\nabla^\pi} (d^\pi w)= T^\pi(\Pi dw, \Pi dw)+ w^*\lambda \wedge \left(\frac{1}{2} (\CL_{X_\lambda}J)\, Jd^\pi w\right).
\ee
\end{lem}
\begin{proof}
For given $\xi_1, \, \xi_2 \in \Gamma(T\Sigma)$, evaluate $d^{\nabla^\pi} (d^\pi w)(\xi_1, \xi_2)$ as
\beastar
&{}& d^{\nabla^\pi} (d^\pi w)(\xi_1, \xi_2)\\
&=&(\nabla^\pi_{\xi_1}(\pi dw))(\xi_2)-(\nabla^\pi_{\xi_2}(\pi dw))(\xi_1)\\
&=&\left(\nabla^\pi_{\xi_1}(\pi dw(\xi_2))-\pi dw(\nabla_{\xi_1}\xi_2)\right)
-\left(\nabla^\pi_{\xi_2}(\pi dw(\xi_1))
-\pi dw\left(\nabla_{\xi_2}\xi_1\right)\right)\\
&=& \pi\Big(\left(\nabla_{\xi_1}( dw(\xi_2))-\nabla_{\xi_1}(\lambda(dw(\xi_2))X_\lambda)\right)
-\left(\nabla_{\xi_2}( dw(\xi_1))-\nabla_{\xi_2}(\lambda(dw(\xi_1))X_\lambda)\right)\\
&{}&- dw \left(\nabla_{\xi_1}\xi_2-\nabla_{\xi_1}\xi_2\right)\Big)\\
&=&\pi \left(\nabla_{\xi_1}(dw(\xi_2))-\nabla_{\xi_2}(dw(\xi_1))-[dw(\xi_1), dw(\xi_2)]\right)\\
&{}&-\nabla_{\xi_1}(\lambda(dw(\xi_2))X_\lambda)
+\nabla_{\xi_2}(\lambda(dw(\xi_1))X_\lambda)\Big)\\
&=& \pi(T(dw(\xi_1),  dw(\xi_2))-\lambda(dw(\xi_2))\nabla_{\xi_1}X_\lambda-\xi_1[\lambda(dw(\xi_2))]X_\lambda\\
&{}& +\lambda(dw(\xi_1))\nabla_{\xi_2}X_\lambda + \xi_2[\lambda(dw(\xi_1))]X_\lambda\Big) \\
&=&\pi(T(dw(\xi_1),  dw(\xi_2)))-\lambda(dw(\xi_2))\nabla_{\xi_1}X_\lambda+\lambda(dw(\xi_1))\nabla_{\xi_2}X_\lambda\\
&=&T^\pi(\Pi dw(\xi_1), \Pi dw(\xi_2))\\
&{}&+\frac{1}{2}\lambda(dw(\xi_2))J(\CL_{X_\lambda}J)\pi dw(\xi_1)-\frac{1}{2}\lambda(dw(\xi_1))J(\CL_{X_\lambda}J)\pi dw(\xi_2)\\
&=&T^\pi(\Pi dw(\xi_1), \Pi dw(\xi_2))\\
&{}&-\frac{1}{2}\lambda(dw(\xi_2))(\CL_{X_\lambda}J)J\pi dw(\xi_1)+\frac{1}{2}\lambda(dw(\xi_1))(\CL_{X_\lambda}J)J\pi dw(\xi_2).
\eeastar
Here we used \eqref{eq:nablaYX} and Axiom (3) for the last second equality.
Rewrite the above result as
$$
d^{\nabla^\pi} (d^\pi w)= T^\pi(\Pi dw, \Pi dw)+ w^*\lambda \wedge \left(\frac{1}{2} (\CL_{X_\lambda}J)\, Jd^\pi w\right)
$$
for any $w$, and we have finished the proof.
\end{proof}

We warn that readers should not get confused with the wedge product we have used here, which is
the wedge product for forms in the usual sense, i.e., $(\alpha_1\otimes\zeta)\wedge\alpha_2=(\alpha_1\wedge\alpha_2)\otimes\zeta$
for $\alpha_1, \alpha_2\in \Omega^*(P)$ and $\zeta$ a section of $E$. This is not the wedge product defined in Appendix \ref{appen:forms}.

We now restrict the above lemma to the case of
contact Cauchy--Riemmann maps, i.e., maps satisfying $\delbar^\pi w = 0$.
In \cite[Theorem 7.3.4]{oh:book} the author proves that
for any standard $J$-holomorphic map $u$ in an almost K\"ahler manifold, the $u^*TM$-valued one-form $du$ is harmonic with respect to the canonical connection.
Now the following Theorem \ref{thm:Laplacian-w}
is its contact analogue, which describes how much $d^\pi w(=\del^\pi w)$ deviates from being a $w^*\xi$-valued harmonic one-form.
The formula explicitly calculates
the difference, which is caused by Reeb projection  and
corresponds to the second term of \eqref{eq:dnabladpiw}.

As an immediate corollary of the previous lemma applied to the contact Cauchy--Riemann maps,
we derive the following formula, calling it the \emph{fundamental equation}.

\begin{thm}[Fundamental Equation]\label{thm:Laplacian-w}
Let $w$ be a contact Cauchy--Riemann map, i.e., a solution of $\delbar^\pi w=0$.
Then
\be\label{eq:Laplacian-w}
d^{\nabla^\pi} (d^\pi w) =d^{\nabla^\pi} (\del^\pi w)= -w^*\lambda\circ j \wedge\left(\frac{1}{2} (\CL_{X_\lambda}J)\, \del^\pi w\right).
\ee
\end{thm}
\begin{proof} The first equality follows since
$d^\pi w=\del^\pi w$ for the solution $w$. Also notice that being a contact Cauchy--Riemann map implies that
$$
T^\pi(\Pi dw,\Pi dw) = T^\pi(\del^\pi w, \del^\pi w) = 0,
$$
which is due to the torsion $T^\pi |_\xi$ being of $(0,2)$-type (in particular, having vanishing $(2,0)$-component).
Furthermore we write \eqref{eq:dnabladpiw} as
\beastar
d^{ \nabla^\pi}(d^\pi w) & = & w^*\lambda
\wedge \left(\frac{1}{2} (\CL_{X_\lambda}J)\, J\del^\pi w\right)\\
& = &  w^*\lambda \wedge \left(\frac{1}{2} (\CL_{X_\lambda}J)\, \del^\pi w\right) \circ j\\
&=&-w^*\lambda\circ j \wedge\left(\frac{1}{2} (\CL_{X_\lambda}J)\, \del^\pi w\right),
\eeastar
using the identity $J\del^\pi w = \del^\pi w \circ j$.
\end{proof}

\begin{cor}[Fundamental Equation in Isothermal Coordinates]
Let $(\tau,t)$ be an isothermal coordinates.
Write $\zeta := \pi \frac{\del w}{\del \tau}$
as a section of $w^*\xi \to M$. Then
\be\label{eq:main-eq-a}
\nabla_\tau^\pi \zeta + J \nabla_t^\pi \zeta
 - \frac{1}{2} \lambda(\frac{\del w}{\del t})(\CL_{X_\lambda}J)\zeta + \frac{1}{2} \lambda(\frac{\del w}{\del \tau})(\CL_{X_\lambda}J)J\zeta =0.
\ee
\end{cor}
\begin{proof}
We denote $\pi \frac{\del w}{\del t}$ by $\eta$. By the isothermality of the coordinate $(\tau,t)$, we
have $J \frac{\del}{\del \tau} = \frac{\del}{\del t}$. Using the $(j,J)$-linearity of $d^\pi w$, we derive
$$
\eta = dw^\pi (\frac{\del}{\del t}) = dw^\pi ( j\frac{\del}{\del \tau}) = J dw^\pi (\frac{\del}{\del \tau}) = J\zeta.
$$

Now we evaluate each side of \eqref{eq:Laplacian-w} against $(\frac{\del}{\del\tau}, \frac{\del}{\del t})$.
For the  left hand side, we get
$$
\nabla^\pi_\tau\eta -\nabla^\pi_t \zeta = \nabla^\pi_\tau J\zeta -\nabla^\pi_t \zeta = J\nabla^\pi_\tau \zeta -\nabla^\pi_t \zeta.
$$
For the right hand side, we get
\beastar
&{}& \frac{1}{2} \lambda(\frac{\del w}{\del \tau})(\CL_{X_\lambda}J) J\eta
-\frac{1}{2} \lambda(\frac{\del w}{\del t})(\CL_{X_\lambda}J) J\zeta\\
& = & - \frac{1}{2} \lambda(\frac{\del w}{\del \tau})(\CL_{X_\lambda}J) \zeta
-\frac{1}{2} \lambda(\frac{\del w}{\del t})(\CL_{X_\lambda}J) J\zeta
\eeastar
where we use the equation $\eta = J \zeta$ for the equality. By setting them equal and applying $J$
to the resulting equation using the fact that $\CL_{X_\lambda}J$ anti-commutes with $J$, we
obtain the equation.
\end{proof}

The fundamental equation in cylindrical coordinates $(\tau, t)\in [0,\infty)\times S^1$
plays an important role in the derivation of the exponential decay of the derivatives at cylindrical ends.
(See Part II of \cite{oh-wang2}.)

\begin{rem}
The fundamental equation in cylindrical coordinates is nothing but the linearization of the
contact Cauchy--Riemann equation in the direction $\frac{\del}{\del\tau}$.
\end{rem}

The following lemmas  will be needed in the calculation of
$\langle \Delta^{\nabla^\pi} d^\pi w, d^\pi w\rangle$ for contact Cauchy--Riemann maps
$d^\pi w=\del^\pi w$.

\begin{lem}\label{lem:2delta}
For any smooth map $w$, we have
\beastar
\langle d^{\nabla^\pi}  \delta^{\nabla^\pi}\del^\pi w, \del^\pi w\rangle=\langle \delta^{\nabla^\pi}d^{\nabla^\pi}  \del^\pi w, \del^\pi w\rangle.
\eeastar
As a consequence,
\bea
\langle \Delta^{\nabla^\pi} \del^\pi w, \del^\pi w\rangle=2\langle \delta^{\nabla^\pi}d^{\nabla^\pi}  \del^\pi w, \del^\pi w\rangle.\label{eq:2Delta}
\eea
\end{lem}
\begin{proof}
\bea
\langle \delta^{\nabla^\pi}d^{\nabla^\pi}  \del^\pi w, \del^\pi w\rangle &=&-\langle *d^{\nabla^\pi}  * d^{\nabla^\pi}  \del^\pi w, \del^\pi w\rangle\nonumber\\
&=&-\langle d^{\nabla^\pi} * d^{\nabla^\pi}  \del^\pi w, *\del^\pi w\rangle\nonumber\\
&=&-\langle d^{\nabla^\pi} * d^{\nabla^\pi}  \del^\pi w, -\del^\pi w\circ j\rangle\label{eq:dstard1}\\
&=&\langle d^{\nabla^\pi} * d^{\nabla^\pi}  \del^\pi w, J \del^\pi w\rangle\nonumber\\
&=&-\langle J d^{\nabla^\pi} * d^{\nabla^\pi}  \del^\pi w,  \del^\pi w\rangle\nonumber\\
&=&-\langle d^{\nabla^\pi} * d^{\nabla^\pi}  J\del^\pi w,  \del^\pi w\rangle\label{eq:dstard3}\\
&=&-\langle d^{\nabla^\pi}  * d^{\nabla^\pi}  \del^\pi w\circ j,  \del^\pi w\rangle\nonumber\\
&=&\langle d^{\nabla^\pi} * d^{\nabla^\pi} * \del^\pi w,  \del^\pi w\rangle\label{eq:dstard5}\\
&=&\langle d^{\nabla^\pi}  \delta^{\nabla^\pi}\del^\pi w, \del^\pi w\rangle.\nonumber
\eea
Here for \eqref{eq:dstard1} and \eqref{eq:dstard5}, we use $*\alpha=-\alpha\circ j$ for any $1$-form $\alpha$.
For \eqref{eq:dstard3}, we use the fact that the connection is $J$-linear.
\end{proof}

The following formula expresses $\langle\Delta^{\nabla^\pi}d^\pi w, d^\pi w\rangle$, which involves the third derivative of $w$,
in terms of terms involving derivatives of order at most two.

\begin{lem}\label{lem:laplaceproduct}
For any contact Cauchy--Riemann map $w$,
\be\label{eq:-Laplaciandpiw}
-\langle \Delta^{\nabla^\pi}d^\pi w, d^\pi w\rangle
= \langle\delta^{\nabla^\pi}[(w^*\lambda\circ j)\wedge (\CL_{X_\lambda}J)\del^\pi w], d^\pi w\rangle.
\ee
Furthermore we can write
\bea\label{eq:deltapi<>}
&{}&
\delta^{\nabla^\pi}[(w^*\lambda\circ j)\wedge (\CL_{X_\lambda}J)\del^\pi w]\nonumber\\
&=&-*\langle (\nabla^\pi(\CL_{X_\lambda}J))\del^\pi w, w^*\lambda\rangle \nonumber\\
&&
-*\langle (\CL_{X_\lambda}J)\nabla^\pi\del^\pi w, w^*\lambda\rangle
-*\langle (\CL_{X_\lambda}J)\del^\pi w, \nabla w^*\lambda\rangle.
\eea
\end{lem}
\begin{proof} The first equality \eqref{eq:-Laplaciandpiw} immediately follows from the fundamental equation,
Theorem \ref{thm:Laplacian-w}, and \eqref{eq:2Delta} of Lemma \ref{lem:2delta}  for contact Cauchy--Riemann maps.

For the second equality \eqref{eq:deltapi<>}, using the identities $\delta^{\nabla^\pi} = - *d^{\nabla^\pi}*$
for two-forms and $*\alpha = -\alpha \circ j$ for general one-form $\alpha$, we rewrite
\beastar
\delta^{\nabla^\pi}[(w^*\lambda\circ j)\wedge (\CL_{X_\lambda}J)\del^\pi w]=
 -*d^{\nabla^\pi}*[(\CL_{X_\lambda}J)\del^\pi w\wedge (*w^*\lambda)],
\eeastar
and then apply the definition of the Hodge $*$ (see Appendix \ref{appen:forms})
to the expression $*[(\CL_{X_\lambda}J)\del^\pi w\wedge (*w^*\lambda)]$, and get
\beastar
&&\delta^{\nabla^\pi}[(w^*\lambda\circ j)\wedge (\CL_{X_\lambda}J)\del^\pi w]\\
&=&-*d^{\nabla^\pi}\langle (\CL_{X_\lambda}J)\del^\pi w, w^*\lambda\rangle\\
&=&-*\langle (\nabla^\pi(\CL_{X_\lambda}J))\del^\pi w, w^*\lambda\rangle
-*\langle (\CL_{X_\lambda}J)\nabla^\pi\del^\pi w, w^*\lambda\rangle
-*\langle (\CL_{X_\lambda}J)\del^\pi w, \nabla w^*\lambda\rangle.
\eeastar
This finishes the proof.
\end{proof}

Here in the above lemma $\langle\cdot, \cdot\rangle$ denotes the inner product induced from $h$, i.e.,
$\langle\alpha_1\otimes\zeta, \alpha_2\rangle:=h(\alpha_1, \alpha_2)\zeta$,
for any $\alpha_1, \alpha_2\in \Omega^k(P)$ and $\zeta$ a section of $E$. This inner product should not be confused with the inner product of the vector bundles.

By applying $\delta^{\nabla^\pi}$ to \eqref{eq:Laplacian-w} and the resulting expression of
$\Delta^{\nabla^\pi}(d^\pi w) = \Delta^{\nabla^\pi}(\del^\pi w)$ thereinto and \eqref{eq:-Laplaciandpiw}, we can convert
the Weitzenb\"ock formula \eqref{eq:bochner-weitzenbock-e} into
\bea\label{eq:e-pi-weitzenbock}
-\frac{1}{2}\Delta e^\pi(w)&=&|\nabla^\pi (\del^\pi w)|^2+K|\del^\pi w|^2+\langle \ric^{\nabla^\pi} (\del^\pi w), \del^\pi w\rangle\nonumber\\
&{}&+\langle \delta^{\nabla^{\pi}}[(w^*\lambda\circ j)\wedge (\CL_{X_\lambda}J)\del^\pi w], \del^\pi w\rangle
\eea
for any contact Cauchy--Riemann map, i.e., any map $w$ satisfying $\delbar^\pi w = 0$.

\section{A priori estimates for contact instantons}

In this section, we derive some basic estimates for the (full) energy density
$|dw|^2$ of contact instantons $w$.
These estimates are important for the derivation of local regularity and
$\epsilon$-regularity needed for the compactification of certain moduli space.
(See \cite{oh:sigmamodel} for  further study along this lines.)

\subsection{$W^{2,2}$-estimates}

Recall from the last section that we have derived the following identity
\bea\label{eq:e-pi-weitzenbock}
-\frac{1}{2}\Delta e^\pi(w)&=&|\nabla^\pi (\del^\pi w)|^2+K|\del^\pi w|^2+\langle \ric^{\nabla^\pi} (\del^\pi w), \del^\pi w\rangle
\nonumber\\
&{}&+\langle \delta^{\nabla^\pi}[(w^*\lambda\circ j)\wedge (\CL_{X_\lambda}J)\del^\pi w], \del^\pi w\rangle.
\eea

By \eqref{eq:deltapi<>}, the first entry in
$\langle \delta^{\nabla^\pi}[(w^*\lambda\circ j)\wedge (\CL_{X_\lambda}J)\del^\pi w], \del^\pi w\rangle$
can be written as
\bea\label{eq:delta-ident}
&{}&\delta^{\nabla^\pi}[(w^*\lambda\circ j)\wedge (\CL_{X_\lambda}J)\del^\pi w]\nonumber\\
&=&-*\langle (\nabla^\pi(\CL_{X_\lambda}J))\del^\pi w, w^*\lambda\rangle
-*\langle (\CL_{X_\lambda}J)\nabla^\pi\del^\pi w, w^*\lambda\rangle
-*\langle (\CL_{X_\lambda}J)\del^\pi w, \nabla w^*\lambda\rangle.\nonumber\\
\eea

Hence we get a bound for the last term $\langle \delta^{\nabla^\pi}[(w^*\lambda\circ j)\wedge (\CL_{X_\lambda}J)\del^\pi w], \del^\pi w\rangle$
by
\beastar
&{}&|\langle \delta^{\nabla^\pi}[(w^*\lambda\circ j)\wedge (\CL_{X_\lambda}J)\del^\pi w], \del^\pi w\rangle|\\
&\leq&
\|\nabla^\pi(\CL_{X_\lambda}J)\|_{C^0(M)}|dw|^4\\
&{}&+|\langle (\CL_{X_\lambda}J)\nabla^\pi(\del^\pi w), w^*\lambda\rangle||dw|
+|\langle (\CL_{X_\lambda}J)\del^\pi w, \nabla w^*\lambda\rangle||dw|.
\eeastar
We further bound the last two terms of \eqref{eq:delta-ident} via
\beastar
|\langle (\CL_{X_\lambda}J)\nabla^\pi(\del^\pi w), w^*\lambda\rangle||dw|&\leq&
\|\CL_{X_\lambda}J\|_{C^0(M)}|\nabla^\pi (\del^\pi w)||dw|^2\\
&\leq&\frac{1}{2c}|\nabla^\pi (\del^\pi w)|^2+\frac{c}{2}\|\CL_{X_\lambda}J\|^2_{C^0(M)}|dw|^4\\
\text{and }\quad \quad \quad\quad\quad\quad\quad\quad\quad\quad\quad\quad\quad\quad&{}&\\
|\langle (\CL_{X_\lambda}J)\del^\pi w, \nabla w^*\lambda\rangle||dw|
&\leq&\frac{1}{2c}|\nabla w^*\lambda|^2+\frac{c}{2}\|\CL_{X_\lambda}J\|^2_{C^0(M)}|dw|^4
\eeastar
similarly. Here $c$ is any positive constant.

Finally, we get the upper bound for
\bea\label{eq:laplacian-pi-upper}
&{}&|\langle \delta^{\nabla^\pi}[(w^*\lambda\circ j)\wedge (\CL_{X_\lambda}J)\del^\pi w], \del^\pi w\rangle|\nonumber\\
&\leq&
\frac{1}{2c}\left(|\nabla^\pi (\del^\pi w)|^2+ |\nabla w^*\lambda|^2\right)
+\left(c\|\CL_{X_\lambda}J\|^2_{C^0(M)}+\|\nabla^\pi(\CL_{X_\lambda}J)\|_{C^0(M)}\right)|dw|^4\nonumber\\
\eea
for any contact Cauchy--Riemann map $w$.

Now we consider contact instantons which are Cauchy--Riemann maps satisfying $\delta w^*\lambda=0$ in addition.
Using the Bochner--Weitzenb\"ock formula (applied to differential forms on a Riemann surface), we get the following identity
\be\label{eq:e-lambda-weitzenbock}
-\frac{1}{2}\Delta|w^*\lambda|^2=|\nabla w^*\lambda|^2+K|w^*\lambda|^2-\langle \Delta (w^*\lambda), w^*\lambda\rangle.
\ee
Write
$$
\Delta (w^*\lambda)=d\delta (w^*\lambda)+\delta d(w^*\lambda),
$$
in which the first term vanishes since $\delta w^*\lambda = - d(w^*\lambda \circ j) = 0$.
Then
\beastar
\langle \Delta (w^*\lambda), w^*\lambda\rangle&=&\langle \delta d(w^*\lambda), w^*\lambda\rangle\\
&=&-\frac{1}{2}\langle *d|\del^\pi w|^2, w^*\lambda\rangle\\
&=&-\langle *\langle \nabla^\pi \del^\pi w, \del^\pi w\rangle,  w^*\lambda\rangle.
\eeastar
Similarly as in the previous estimates for the Laplacian term of $\del^\pi w$, we can bound
\bea\label{eq:laplacian-lambda-upper}
|-\langle \Delta (w^*\lambda), w^*\lambda\rangle|&=&|\langle *\langle \nabla^\pi \del^\pi w, \del^\pi w\rangle,  w^*\lambda\rangle|\nonumber\\
&\leq& |\nabla^\pi \del^\pi w||dw|^2\nonumber\\
&\leq& \frac{1}{2c}|\nabla^\pi \del^\pi w|^2+\frac{c}{2}|dw|^4.
\eea

At last, we calculate the total energy density which is defined as
$$
e(w):=|dw|^2=e^\pi(w)+|w^*\lambda|^2.
$$
Summing up \eqref{eq:e-pi-weitzenbock} and \eqref{eq:e-lambda-weitzenbock}, and applying the estimates \eqref{eq:laplacian-pi-upper}
and \eqref{eq:laplacian-lambda-upper} respectively, we obtain the following inequality for any contact instanton $w$
\bea\label{eq:laplace-e-derivative}
&{}& -\frac{1}{2}\Delta e(w)\nonumber\\
&\geq& \left(1-\frac{1}{c}\right)|\nabla^\pi(\del^\pi w)|^2+\left(1-\frac{1}{2c}\right)|\nabla w^*\lambda|^2\nonumber\\
&{}&- \left(c\|\CL_{X_\lambda}J\|^2_{C^0(M)}+\|\nabla^\pi(\CL_{X_\lambda}J)\|_{C^0(M)}+\frac{c}{2}+\|\ric\|_{C^0(M)} \right)e(w)^2
+Ke(w)\nonumber\\
\\
&\geq& - \left(c\|\CL_{X_\lambda}J\|^2_{C^0(M)}+\|\nabla^\pi(\CL_{X_\lambda}J)\|_{C^0(M)}
+\frac{c}{2}+\|\ric\|_{C^0(M)} \right)e(w)^2+Ke(w),\nonumber
\eea
for any $c>1$. We fix $c=2$ and get the following

\begin{thm} For a contact instanton $w$, we have the following differential inequality
$$
\Delta e(w)\leq Ce(w)^2+\|K\|_{L^\infty(\dot\Sigma)}e(w),
$$
where
$$
C=2\|\CL_{X_\lambda}J\|^2_{C^0(M)}+\|\nabla^\pi(\CL_{X_\lambda}J)\|_{C^0(M)}+\|\ric\|_{C^0(M)}+1
$$
which is a positive constant independent of $w$.
\end{thm}

Once we have this differential inequality,  we obtain the following interior density estimates
by the standard  argument from \cite{schoen}. (Also see the proof of \cite[Theorem 8.1.3]{oh:book}
given in the context of pseudoholomorphic curves.)
\begin{cor}[$\epsilon$-regularity and interior density estimate]\label{density}
There exist constants $C, \, \e_0$ and $r_0 > 0$, depending only on $J$ and
the Hermitian metric $h$ on $\dot \Sigma$, such that for any
 $C^1$ contact instanton $w: \dot \Sigma \to M$ with
$$
E(r_0): = \frac{1}{2}\int_{D(r_0)} |dw|^2 \leq \e_0,
$$
and discs $D(2r) \subset \operatorname{Int}\Sigma$ with $0 < 2r \leq r_0$,
$w$ satisfies
\be\label{eq:schoen's}
\max_{\sigma \in (0,r]} \left(\sigma^2 \sup_{D(r-\sigma)}
e(w)\right) \leq CE(r)
\ee
for all $0< r \leq r_0$. In particular, letting $\sigma = r/2$, we obtain
\be\label{eq:supeu}
\sup_{D(r/2)} |dw|^2 \leq \frac{4C E(r)}{r^2}
\ee
for all $r \leq r_0$.
\end{cor}

Now we rewrite $\eqref{eq:laplace-e-derivative}$ into
\bea\label{eq:laplace-higherderivative}
&{}&\left(1-\frac{1}{c}\right)|\nabla^\pi(\del^\pi w)|^2+\left(1-\frac{1}{2c}\right)|\nabla w^*\lambda|^2\nonumber\\
&\leq&-\frac{1}{2}\Delta e(w) - Ke(w) \nonumber\\
& {}& \quad +\left(c\|\CL_{X_\lambda}J\|^2_{C^0(M)}+\|\nabla^\pi(\CL_{X_\lambda}J)\|_{C^0(M)}+\frac{c}{2}+\|\ric\|_{C^0(M)} \right)e^2
\eea

We want to get a coercive $L^2$ bound for $\nabla dw$, which consists of the two parts given below according to the decomposition
$dw=d^\pi w+w^*\lambda\otimes X_\lambda$.
\be\label{eq:|nabladw|2-1}
|\nabla dw|^2 = |\nabla(d^\pi w)+ \nabla(w^*\lambda\otimes X_\lambda)|^2
\leq 2 |\nabla(d^\pi w)|^2 + 2 |\nabla(w^*\lambda\otimes X_\lambda)|^2.
\ee
For the first term on the right hand side of \eqref{eq:|nabladw|2-1}, we write
\bea
|\nabla(d^\pi w)|^2&=&|\nabla^\pi (d^\pi w)|^2+|\langle \nabla(d^\pi w), X_\lambda\rangle|^2\nonumber\\
&=&|\nabla^\pi (d^\pi w)|^2+\frac{1}{4}|\langle d^\pi w, (\CL_{X_\lambda}J)J d^\pi w\rangle|^2\label{eq:nabla-dpi1}\\
&\leq&|\nabla^\pi (d^\pi w)|^2+\frac{1}{4}|\CL_{X_\lambda}J|^2|d^\pi w|^4\nonumber\\
&\leq&|\nabla^\pi (d^\pi w)|^2+\frac{1}{4}\|\CL_{X_\lambda}J\|^2_{C^0(M)}|d^\pi w|^4,\nonumber
\eea
where \eqref{eq:nabla-dpi1} comes from the metric property of the contact triad connection together with
\eqref{eq:nablaYX}.

For the second term on the right hand side of \eqref{eq:|nabladw|2-1}, we again apply \eqref{eq:nablaYX}
and derive
\beastar
|\nabla(w^*\lambda\otimes X_\lambda)|^2&=&|(\nabla w^*\lambda)\otimes X_\lambda+(w^*\lambda)\otimes\nabla X_\lambda|^2\\
&=&|\nabla w^*\lambda|^2+|w^*\lambda|^2|\frac{1}{2}(\CL_{X_\lambda}J)Jd^\pi w|^2\\
&\leq&|\nabla w^*\lambda|^2+\frac{1}{4}\|\CL_{X_\lambda}J\|^2_{C^0(M)}|w^*\lambda|^2|d^\pi w|^2.
\eeastar

Summing up the two terms and going back to \eqref{eq:|nabladw|2-1}, we get
\bea\label{eq:higherorder1}
|\nabla(dw)|^2&\leq& 2|\nabla^\pi (d^\pi w)|^2+2|\nabla w^*\lambda|^2\nonumber\\
&{}&+\frac{1}{2}\|\CL_{X_\lambda}J\|^2_{C^0(M)}|d^\pi w|^4+\frac{1}{2}\|\CL_{X_\lambda}J\|^2_{C^0(M)}|w^*\lambda|^2|d^\pi w|^2.
\eea

Hence from this, we have
\beastar
|\nabla(dw)|^2&\leq&
\frac{2}{1-\frac{1}{c}}\left[\left(1-\frac{1}{c}\right)|\nabla^\pi(\del^\pi w)|^2+\left(1-\frac{1}{2c}\right)|\nabla w^*\lambda|^2 \right]\\
&{}&+\|\CL_{X_\lambda}J\|^2_{C^0(M)}|d w|^4
\eeastar
and combine it with \eqref{eq:laplace-higherderivative}, we get
\beastar
& {}& |\nabla(dw)|^2\\
&\leq&
\left[(\frac{2c^2}{c-1}+1)\|\CL_{X_\lambda}J\|^2_{C^0(M)}+\frac{2c}{c-1}
\left(\|\nabla^\pi(\CL_{X_\lambda}J)\|_{C^0(M)}+\frac{c}{2}+\|\ric\|_{C^0(M)} \right)\right]|dw|^4\\
&{}&-\frac{2c\cdot K}{c-1}|dw|^2
+\frac{c}{1-c}\Delta e(w)
\eeastar
for any constant $c>1$. We still take $c=2$ and get the following coercive estimate for contact instantons
\be\label{eq:higher-derivative}
|\nabla(dw)|^2\leq C_1|dw|^4-4K|dw|^2-2\Delta e(w)
\ee
where
$$
C_1:=9\|\CL_{X_\lambda}J\|^2_{C^0(M)}+4\|\nabla^\pi(\CL_{X_\lambda}J)\|_{C^0(M)}+4\|\ric\|_{C^0(M)}+4
$$
denotes a constant.

The following local a priori estimate can be easily derived from \eqref{eq:higher-derivative}
by the standard usage of cut-off function. We give its proof in Appendix \ref{appen:local-coercive}.
\begin{prop}\label{prop:coercive-L2}
For any pair of domains $D_1$ and $D_2$ in $\dot\Sigma$ such that $\overline{D_1}\subset D_2$,
$$
\|\nabla(dw)\|^2_{L^2(D_1)}\leq C_1(D_1, D_2)\|dw\|^2_{L^2(D_2)}+C_2(D_1, D_2)\|dw\|^4_{L^4(D_2)}
$$
for any contact instanton $w$,
where $C_1(D_1, D_2)$, $C_2(D_1, D_2)$ are some constants which depend on $D_1$, $D_2$ and $(M, \lambda, J)$, but are independent of $w$.
\end{prop}

We remark that this proposition is nothing but a re-statement of Theorem \ref{thm:local-W12} in the introduction.

\subsection{Local $W^{2+k,2}$ estimates for $k \geq 1$}

Starting from the above $W^{2,2}$-estimate, we proceed to higher $W^{2+k,2}$-estimates
inductively. For this purpose, we consider the decomposition
$$
dw = d^\pi w + w^*\lambda\otimes X_\lambda
$$
and estimate $|\nabla^{k+1} d w|$ inductively staring from $k=0$ which is done in the previous subsection.

The rest of this subsection will be occupied by the proof of the following theorem.

\begin{thm}\label{thm:local-higher-regularity}
Let $w$ be a contact instanton.
Then for any pair of domains $D_1 \subset D_2 \subset \dot \Sigma$ such that $\overline{D_1}\subset D_2$, we have
$$
\int_{D_1} |(\nabla)^{k+1}(dw)|^2 \leq \int_{D_2} \CJ_{k}(d^\pi w, w^*\lambda).
$$
Here $\CJ_k$ is a polynomial function of degree up to $2k+4$ with nonnegative coefficients  of the norms of the covariant derivatives
of $d^\pi w, \, w^*\lambda$ up to $0, \, \ldots, k$ with degree at most $2k + 4$
whose coefficients depending on $J$, $\lambda$ and $D_1, \, D_2$ but independent of $w$.
\end{thm}

We start with the following lemma
\begin{lem}\label{lem:nablad}For any $k\geq 0$,
$$
\nabla^{k+1}dw=(\nabla^\pi)^{k+1}d^\pi w+\nabla^{k+1} w^*\lambda\otimes X_\lambda
+O_{k}(d^\pi w, w^*\lambda),
$$
where $O_k(d^\pi w, w^*\lambda)$ denotes some tensor living in $\CT^{k+1}\otimes \Omega^1(w^*TM) \subset \CT^{k+1}_1(w^*TM)$.
More specifically $O_{k}(d^\pi w, w^*\lambda)$
can be written into the form of a polynomial which consists of monomials of one of the following forms
\beastar
&&a \cdot (\bigotimes_{i=1, \cdots, |m|}(\nabla^\pi)^{m_i} d^\pi w\otimes  \bigotimes_{j=1, \cdots, |n|}\nabla^{n_j} w^*\lambda)\otimes d^\pi w, \\
&& b\cdot (\bigotimes_{i'=1, \cdots, |m|'}(\nabla^\pi)^{m'_{i'}} d^\pi w\otimes  \bigotimes_{j'=1 ,\cdots, |n|'}\nabla^{n'_{j'}} w^*\lambda)\otimes X_\lambda(w)
\eeastar
with $i, j, i', j'$, $m_i, n_j, m'_{i'}, n'_{j'}\leq k$ and
$$
1 \leq \Sigma_{i}m_i+\Sigma_{j}n_j\leq k+1,  \quad 2 \leq \Sigma_{i'}m'_{i'}+\Sigma_{j'}n'_{j'}\leq k+1
$$
and $a, b$ are some $C^\infty$ bounded functions on $\dot\Sigma$.
\end{lem}
\begin{proof}For the case $k=0$, we compute
\beastar
\nabla dw&=&\nabla d^\pi w+\nabla (w^*\lambda\otimes X_\lambda)\\
& = &  \nabla^\pi d^\pi w + \langle \nabla (d^\pi w), X_\lambda \rangle X_\lambda
 + \nabla w^*\lambda \otimes X_\lambda + w^*\lambda \otimes \nabla X_\lambda \\
&=&\nabla^\pi d^\pi w - \langle d^\pi w, \nabla X_\lambda \rangle X_\lambda
 + (\nabla w^*\lambda) \otimes X_\lambda + w^*\lambda \otimes \nabla X_\lambda \\
&=&\nabla^\pi d^\pi w+(\nabla w^*\lambda)\otimes X_\lambda(w)\\
&+&w^*\lambda\otimes \frac{1}{2}(\CL_{X_\lambda}J)J d^\pi w
- \left\langle d^\pi w, \frac{1}{2}(\CL_{X_\lambda}J)J d^\pi w \right \rangle\otimes X_\lambda.
\eeastar
It is obviously of the form in our induction assumption with the help of the metric tensor over $M$.
(Here $|m|=0$, $|n|=1$, $n_1=1$, and $|m|'=2$, $m'_1=m'_2=1$, $|n|'=0$.)

Now assuming the expression for any $0\leq i\leq k$ with $k\geq 0$ holds, we show that it holds for $k+1$ too.
First by the induction hypothesis, $\nabla^{k+1}dw$ can be decomposed into
\beastar
\nabla^{k+1}dw&=&\nabla(\nabla^k dw)\\
&=&\nabla((\nabla^\pi)^{k}d^\pi w)+\nabla(\nabla^{k} w^*\lambda\otimes X_\lambda)
+\nabla O_{k-1}(d^\pi w, w^*\lambda).
\eeastar
We examine them one by one.
For the term, we compute
\beastar
\nabla((\nabla^\pi)^{k}d^\pi w)&=&(\nabla^\pi)^{k+1}d^\pi w+\langle \nabla((\nabla^\pi)^{k}d^\pi w), X_\lambda \rangle\otimes X_\lambda\\
&=&(\nabla^\pi)^{k+1}d^\pi w-\langle (\nabla^\pi)^{k}d^\pi w,\nabla X_\lambda \rangle\otimes X_\lambda\\
&=&(\nabla^\pi)^{k+1}d^\pi w-\left\langle (\nabla^\pi)^{k}d^\pi w,\frac{1}{2}(\CL_{X_\lambda}J)J d^\pi w \right\rangle\otimes X_\lambda,
\eeastar
where the second term is absorbed into $O_{k}(d^\pi w, w^*\lambda)$.

For the second term, we obtain
\beastar
\nabla(\nabla^{k} w^*\lambda\otimes X_\lambda)&=&\nabla^{k+1} w^*\lambda\otimes X_\lambda
+\nabla^{k} w^*\lambda\otimes \nabla X_\lambda\\
&=&\nabla^{k+1} w^*\lambda\otimes X_\lambda
+\nabla^{k} w^*\lambda\otimes \frac{1}{2}(\CL_{X_\lambda}J)J d^\pi w,
\eeastar
where the second term again goes into $O_{k}(d^\pi w, w^*\lambda)$.

For the third one, we observe that when we take one more derivative of each term $O_{k-1}(d^\pi w, w^*\lambda)$,
the result becomes one of the following six types,
\bea
&&(\nabla a)\otimes (\bigotimes_{i=1, \cdots, |m|}(\nabla^\pi)^{m_i} d^\pi w\otimes  \bigotimes_{j=1, \cdots, |n|}\nabla^{n_j} w^*\lambda)\otimes d^\pi w\label{eq:aterm}\\
&&a\cdot \nabla(\bigotimes_{i=1, \cdots, |m|}(\nabla^\pi)^{m_i} d^\pi w\otimes  \bigotimes_{j=1, \cdots, |n|}\nabla^{n_j} w^*\lambda)\otimes d^\pi w\\
&&a\cdot(\bigotimes_{i=1, \cdots, |m|}(\nabla^\pi)^{m_i} d^\pi w\otimes  \bigotimes_{j=1, \cdots, |n|}\nabla^{n_j} w^*\lambda)\otimes \nabla d^\pi w\\
&&(\nabla b)\otimes(\bigotimes_{i'=1, \cdots, |m|'}(\nabla^\pi)^{m'_{i'}} d^\pi w\otimes  \bigotimes_{j'=1 ,\cdots, |n|'}\nabla^{n'_{j'}} w^*\lambda)\otimes X_\lambda\label{eq:bterm}\\
&&b\cdot \nabla(\bigotimes_{i'=1, \cdots, |m|'}(\nabla^\pi)^{m'_{i'}} d^\pi w\otimes  \bigotimes_{j'=1 ,\cdots, |n|'}\nabla^{n'_{j'}} w^*\lambda)\otimes X_\lambda\\
&&b\cdot (\bigotimes_{i'=1, \cdots, |m|'}(\nabla^\pi)^{m'_{i'}} d^\pi w\otimes  \bigotimes_{j'=1 ,\cdots, |n|'}\nabla^{n'_{j'}} w^*\lambda)\otimes \nabla X_\lambda.
\eea
The \eqref{eq:aterm} and \eqref{eq:bterm} live in $O_k$ because we assume $\nabla a$ (so it $\nabla  b$) can be written as a bounded function tensor along $dw= d^\pi w+w^*\lambda\otimes X_\lambda$.
Other four terms live in $O_k$ because they all raise the order by $1$ either via a direct differentiation or via
a usage of the metric property to rewrite
$$
\nabla(\nabla^\pi)^md^\pi w=\nabla^{m+1}d^\pi w-\langle (\nabla^\pi)^md^\pi w, \nabla X_\lambda\rangle X_\lambda
$$
followed by the insertion $\nabla_{dw}X_\lambda=\frac{1}{2}(\CL_{X_\lambda}J)Jd^\pi w$.

This completes the induction step and hence the proof of the lemma.
\end{proof}
Then applying Proposition \ref{prop:coercive-L2} and using the Cauchy--Schwarz inequality inductively, we immediately get

\begin{cor}\label{cor:kinduction}
For any pair of domains $D_1$ and $D_2$ in $\dot\Sigma$ such that $\overline{D_1}\subset D_2$,
\beastar
\|\nabla^{k+1}dw\|^2_{L^2(D_1)}&\leq& \|(\nabla^\pi)^{k+1}d^\pi w\|^2_{L^2(D_1)}+\|\nabla^{k+1} (w^*\lambda)\|^2_{L^2(D_1)}\\
&+& \int_{D_2}G_k(d^\pi w, w^*\lambda)
\eeastar
for any contact instanton $w$, for another polynomial function of $G_k$ of the type described in Theorem \ref{thm:local-higher-regularity}
\end{cor}

\begin{rem}\label{rem:higher-regularity} Starting from Proposition 5.3, and applying Cauchy--Schwarz inequaltiy and
the induction, we can further obtain the inequality of the form
$$
\int_{D_2}G_k(d^\pi w, w^*\lambda) \leq C_{k;D_1,D_2}(\|dw\|^2_{L^2(D_2)}, \|dw\|^4_{L^2(D_2)})
$$
where $C_{k; D_1,D_2}(r,s)$ is a polynomial function of $r, \,s$ satisfying $C_{k; D_1,D_2}(0,0)=0$.
\end{rem}

Now we estimate $|(\nabla^\pi)^{k+1} d^\pi w|^2 + |\nabla^{k+1}( w^*\lambda)|^2$ inductively.
We first denote
$$
S_k = (\nabla^\pi)^k d^\pi w, \quad T_k=\nabla^k(w^*\lambda).
$$
The general Weitzenb\"ock formula (see (C.7) Appendix \cite{freed-uhlen}
e.g.) applied to $S_k$ and $T_k$ respectively, we obtain
\bea\label{eq:|nablakdpiw|2}
|\nabla^\pi S_k|^2 & = &- \frac{1}{2}\Delta |S_k|^2 + \langle \Delta^\pi S_k, S_k\rangle  - \langle \widetilde R S_k, S_k \rangle\\
|\nabla T_k|^2 & = &- \frac{1}{2}\Delta |T_k|^2 + \langle \Delta T_k, T_k\rangle  - K|T_k|^2.
\eea
where $\widetilde R$ is a zeroth order operator acting on the sections of $w^*\xi \otimes T^*\dot \Sigma$
which depends only on the curvature of the pull-back connection
$w^*\nabla^\pi$ and the Levi-Civita connection of $(\dot \Sigma,h)$. In particular,
$\widetilde R$ is a bounded bilinear form.

Now it remains to prove
\begin{prop}\label{prop:dwpiw*lambda}
For any pair of domains $D_1$ and $D_2$ in $\dot\Sigma$ such that $\overline{D_1}\subset D_2$,
$$
\|\nabla^{k+1}d^\pi w\|^2_{L^2(D_1)}+\|\nabla^{k+1}w^*\lambda\|^2_{L^2(D_1)}
\leq \int_{D_2} M_k(d^\pi w, w^*\lambda)
$$
for any contact instanton $w$, where $M_k$ is another polynomial function of the type described as in Theorem \ref{thm:local-higher-regularity}.
\end{prop}
\begin{proof}
The $k=0$ case is proved by Proposition \ref{prop:coercive-L2}.

For $k \geq 1$, we first quote the following general lemma whose proof is a direct calculation which we leave to the readers.

\begin{lem}\label{lem:commuting} For any $\xi$-valued $1$-form $\alpha$ over the map $w$,
\be\label{eq:dnablapinablazeta}
d^{\nabla^\pi}(\nabla^\pi_{(\cdot)} \alpha) = \nabla^\pi_{(\cdot)}(d^{\nabla^\pi} \alpha) + (R^\pi(dw,dw(\cdot)) \, \alpha)^{\text{\rm skew}}
\ee
where $(R^\pi(dw,dw(\cdot)) \, \alpha)^{\text{skew}}$ is the skew-symmetrization of the bilinear map
$(\xi_1,\xi_2) \mapsto R^\pi(dw(\xi_1),dw(\cdot)) \, \alpha(\xi_2)$, with $R^\pi$ the $\xi$-projection of the curvature of the triad connection $\nabla$.
\end{lem}

Now we choose and fix a domain $D$ and a smooth non-negative cut-off function $\chi:D_2\to \R$,  such that
$\overline D_1 \subset D \subset \overline D \subset D_2$, and $\chi\equiv 1$ on $\overline{D_1}$, $\chi\equiv 0$ on $D_2-D$.
Obviously we have
$$
\int_{D_1} |(\nabla^\pi)^{k+1} d^\pi w)|^2 =
\int_{D_1} |\nabla^\pi S_k|^2\leq  \int_{D} \chi^2 |\nabla^\pi S_k|^2.
$$
On the other hand, applying the Weitzenb\"ock formula similarly as $k=0$,  we write
\bea\label{eq:int}
\int_{D} \chi^2 |\nabla^\pi S_k|^2  =  - \int_{D} \frac{\chi^2}{2} \Delta |S_k|^2 + \int_{D} \chi^2 \langle \Delta^\pi S_k,S_k \rangle -
\int_{D} \chi^2 \langle \widetilde R S_k,S_k  \rangle,\nonumber\\
\eea
for $k\geq 1$,
where $D$ and $\chi$ are chosen the same as in the proof of Proposition \ref{prop:coercive-L2app}.
Obviously the last term  can be bounded by the norm $\|dw\|_{k,2;D_2}^2$, and so
we will focus on the first two terms henceforth.

Similarly as before we get
\bea\label{eq:nablapiSSk}
\int_{D} |\langle \Delta^\pi S_k,S_k \rangle|
 &\leq&  \left(1 + \|d\chi\|_{C^0(D)}\right) \int_{D_2} (|d^{\nabla^\pi}S_k|^2 + |\delta^{\nabla^\pi}S_k|^2)
+ 2  \int_{D_2} |S_k|^2\nonumber\\
&=&2\left(1 + \|d\chi\|_{C^0(D)}\right) \int_{D_2} |d^{\nabla^\pi}S_k|^2
+ 2 \int_{D_2} |S_k|^2, \nonumber\\
&{}&
\eea
where the last equality follows from the $J$-linearity of $\nabla^\pi$ similarly as for Lemma \ref{lem:2delta}.
Again the last term $\int_{D_2} |S_k|^2$ can be bounded by the norm $\|dw\|_{k,2;D_2}^2$, and so it remains to focus on $\int_{D_2} |d^{\nabla^\pi}S_k|^2$.

We first observe the following
\begin{lem} For any $k\geq 0 $, $d^{\nabla^\pi} S_k$ can be written as a sum of tensors of forms $a_{ij}\otimes S_i \otimes T_j$ with $0\leq i, j\leq k$, where $a_{ij}$'s  are some $C^\infty$-bounded sections in $\Omega^1(\dot\Sigma)\otimes w^*TM$.
\end{lem}
\begin{proof}
The proof of this lemma is again by an induction argument.
For $k=0$, we have $S_0 = d^\pi w$ and the fundamental equation \eqref{eq:Laplacian-w}
$$
d^{\nabla^\pi} S_0 = d^{\nabla^\pi}d^\pi w = -\frac{1}{2} w^*\lambda\circ j \wedge (\CL_{X_\lambda}J) d^\pi w.
$$
It can be easily checked
$$
-\frac{1}{2} w^*\lambda\circ j \wedge (\CL_{X_\lambda}J) d^\pi w = \frac{1}{2} w^*\lambda \wedge (\CL_{X_\lambda}J) J d^\pi w.
$$
Combining the two, the initial case $k=0$ holds.

Now suppose the lemma holds for $k-1$ with $k \geq 1$.
Applying Lemma \ref{lem:commuting}, we derive
\bea\label{eq:k-commuting}
d^{\nabla^\pi} S_k & = &  d^{\nabla^\pi} \nabla^\pi S_{k-1} \nonumber\\
& = & \nabla^\pi( d^{\nabla^\pi} S_{k-1}) +(R^\pi(dw,dw(\cdot)) \, S_{k-1})^{\text{\rm skew}}.
\eea
The curvature term is certainly of form required in the lemma (even for $k-1$ instead of $k$) by the
induction hypothesis.

On the other hand, for the first term $\nabla^\pi(d^{\nabla^\pi} S_{k-1})$ in \eqref{eq:k-commuting},
the induction hypothesis implies $d^{\nabla^\pi} S_{k-1}$ is a summand of the terms each of which of the form
$a_{ij}\otimes S_i \otimes T_j$ with $0\leq i, j\leq k-1$. By differentiating this and applying
Lemma \ref{lem:nablad}, we have proved the lemma for $k$. This finishes the proof.
\end{proof}

Using this lemma, we have obtained
$$
\int_{D_2}| d^{\nabla^\pi} S_k|^2 \leq \int_{D_2} H_k(d^\pi w, w^*\lambda),
$$
where $H_k$ is some polynomial function of the type as in Theorem \ref{thm:local-higher-regularity}.
Combining the above two terms in \eqref{eq:nablapiSSk}, we have
obtained the desired polynomial integral bound
$$
\int_{D} |\langle \Delta^\pi S_k,S_k \rangle| \leq I_k(d^\pi w,w^*\lambda)
$$
again with the same kind of polynomial $I_k$,
which in particular implies
\be\label{eq:intchi}
\int_{D} \chi^2 |\langle \Delta^\pi S_k,S_k \rangle| \leq I_k(d^\pi w,w^*\lambda).
\ee

Next we go back to the first term in \eqref{eq:int}, which is
$- \int_{D}\frac{\chi^2}{2}\Delta |(\nabla^\pi)^k d^\pi w|^2$.
For this one, using similar computation as in Appendix \ref{appen:local-coercive},
one can obtain
\bea\label{eq:intchiDelta}
\left| \int_{D} \chi^2 \Delta |S_k|^2 \right|&\leq&
\frac{1}{\epsilon} \int_{D} \chi^2|\nabla^\pi S_k|^2 + \epsilon \int_{D_2} |d\chi|^2|S_k|^2\nonumber\\
&\leq&\frac{1}{\epsilon} \int_{D} \chi^2|\nabla^\pi S_k|^2 + \epsilon \|d\chi\|_{C^0(D)}^2\int_{D_2} |S_k|^2.
\eea
The second term is bounded by a similar polynomial integral bound, which we denote by $I'_k$.
Then by substituting this inequality into \eqref{eq:int}, setting $\epsilon=1$,
using the two polynomial integral bounds from $I_k$ and $I_k'$, and applying a back-substitution,
we obtain
\beastar
\int_{D} \chi^2 |\nabla^\pi S_k|^2
\leq \frac{1}{2} \int_{D}\chi^2 |\nabla^\pi S_k|^2+\int_{D_2}(I_k+I_k')
\eeastar
which is equivalent to
$$
\int_{D}\chi^2 |\nabla^\pi S_k|^2 \leq  2\int_{D_2}(I_k+I_k').
$$
Therefore
we obtain
$$
\int_{D_1}|\nabla^\pi S_k|^2\leq
\int_{D}\chi^2 |\nabla^\pi S_k|^2 \leq  2\int_{D_2}(I_k+I_k').$$
The treatment for $\int_{D_1}|\nabla T_k|^2$ is similar but much simpler, so we omit details.

These together finish the proof of Proposition \ref{prop:dwpiw*lambda}.

\end{proof}

Combining Proposition \ref{prop:dwpiw*lambda} and Corollary \ref{cor:kinduction}, we have proved
Theorem \ref{thm:local-higher-regularity}, where the polynomial $\CJ_k$ can be taken as  the sum of all the polynomials arising from the  proofs of Proposition \ref{prop:dwpiw*lambda} and Corollary \ref{cor:kinduction}. The order of $\CJ_k$ can be limited to $2k+4$ with a careful look at the induction steps.

\medskip

The following is an immediate consequence of Theorem \ref{thm:local-W12}, Theorem \ref{thm:local-higher-regularity}
and Remark \ref{rem:higher-regularity}.

\begin{cor}
Any weak solution of equation \eqref{eq:instanton} in
$W^{1,4}_{\loc}$ automatically lies in $W^{3,2}_{\loc}$ and becomes a classical solution,
hence smooth.
\end{cor}

\section{Asymptotic behavior of contact instantons}
\label{sec:reeborbits}

In this section, we study the asymptotic behavior of contact instantons
on the Riemann surface $(\dot\Sigma, j)$ associated with a metric $h$ with \emph{cylindrical ends}.
To be precise, we assume there exists a compact set $K_\Sigma\subset \dot\Sigma$,
such that $\dot\Sigma-\Int(K_\Sigma)$ is a disjoint union of punctured disks
 each of which is isometric to the half cylinder $[0, \infty)\times S^1$ or $(-\infty, 0]\times S^1$, where
the choice of positive or negative cylinders depends on the choice of analytic coordinates
at the punctures.
We denote by $\{p^+_i\}_{i=1, \cdots, l^+}$ the positive punctures, and by $\{p^-_j\}_{j=1, \cdots, l^-}$ the negative punctures.
Here $l=l^++l^-$. Denote by $\phi^{\pm}_i$ such isometries from cylinders to disks.
We first state our assumptions for the study of the behavior of punctures.

\begin{defn}Let $\dot\Sigma$ be a punctured Riemann surface with punctures
$\{p^+_i\}_{i=1, \cdots, l^+}\cup \{p^-_j\}_{j=1, \cdots, l^-}$ equipped
with a metric $h$ with \emph{cylindrical ends} outside a compact subset $K_\Sigma$.
Let
$w: \dot \Sigma \to M$ be any smooth map. We define the total $\pi$-harmonic energy $E^\pi(w)$
by
\be\label{eq:endenergy}
E^\pi(w) = E^\pi_{(\lambda,J;\dot\Sigma,h)}(w) = \frac{1}{2} \int_{\dot \Sigma} |d^\pi w|^2
\ee
where the norm is taken in terms of the given metric $h$ on $\dot \Sigma$ and the triad metric on $M$.
\end{defn}

We put the following hypotheses in our asymptotic study of the finite
energy contact instanton maps $w$:

\begin{hypo}\label{hypo:basic}
Let $h$ be the metric on $\dot \Sigma$ given above.
Assume $w:\dot\Sigma\to M$ satisfies the contact instanton equations \eqref{eq:instanton},
and
\begin{enumerate}
\item $E^\pi_{(\lambda,J;\dot\Sigma,h)}(w)<\infty$ (finite $\pi$-energy);
\item $\|d w\|_{C^0(\dot\Sigma)} <\infty$.
\end{enumerate}
\end{hypo}

Throughout this section, we work locally near one puncture, i.e., on
$D^\delta(p) \setminus \{p\}$. By taking the associated conformal coordinates $\phi^+ = (\tau,t)
:D^\delta(p) \setminus \{p\} \to [0, \infty)\times S^1 \to $ such that $h = d\tau^2 + dt^2$,
we need only look at a map $w$ defined on the half cylinder $[0, \infty)\times S^1\to M$
without loss of generality.

The above finite $\pi$-energy hypothesis implies
\be\label{eq:hypo-basic-pt}
\int_{[0, \infty)\times S^1}|d^\pi w|^2 \, d\tau \, dt <\infty, \quad \|d w\|_{C^0([0, \infty)\times S^1)}<\infty
\ee
in these coordinates.

Let $w$ satisfy Hypothesis \ref{hypo:basic}. We can associate two
natural asymptotic invariants at each puncture defined as
\bea
T & := & \frac{1}{2}\int_{[0,\infty) \times S^1} |d^\pi w|^2 + \int_{\{0\}\times S^1}(w|_{\{0\}\times S^1})^*\lambda\label{eq:TQ-T}\\
Q & : = & \int_{\{0\}\times S^1}((w|_{\{0\}\times S^1})^*\lambda\circ j).\label{eq:TQ-Q}
\eea
(Here we only look at positive punctures. The case of negative punctures is similar.)

\begin{rem}\label{rem:TQ}
For any contact instanton $w$, since $\frac{1}{2}|d^\pi w|^2\, dA=d(w^*\lambda)$, by Stokes' formula,
$$
T = \frac{1}{2}\int_{[s,\infty) \times S^1} |d^\pi w|^2 + \int_{\{s\}\times S^1}(w|_{\{s\}\times S^1})^*\lambda, \quad
\text{for any } s\geq 0.
$$

Moreover, since $d(w^*\lambda\circ j)=0$, the integral
$$
\int_{\{s \}\times S^1}(w|_{\{s \}\times S^1})^*\lambda\circ j, \quad
\text{for any } s \geq 0
$$
does not depend on $s$ whose common value is nothing but $Q$.
\end{rem}
We call $T$ the \emph{asymptotic contact action}
and $Q$ the \emph{asymptotic contact charge} of the contact instanton $w$ at the given puncture.

For a given contact instanton $w: [0, \infty)\times S^1\to M$, we consider the family of maps
$w_s: [-s, \infty) \times S^1 \to M$ defined by
$
w_s(\tau, t) = w(\tau + s, t)$.
For any compact set $K\subset \R$, there exists some $s_0$ large enough such that
$K\subset [-s, \infty)$ for every $s\geq s_0$. For such $s\geq s_0$, we can also get an $[s_0, \infty)$-family of maps by defining $w^K_s:=w_s|_{K\times S^1}:K\times S^1\to M$.

The asymptotic behavior of $w$ at infinity can be understood by studying the limiting behavior of the sequence of maps
$\{w^K_s:K\times S^1\to M\}_{s\in [s_0, \infty)}$, for each given compact set $K\subset \R$.

First of all,
it is easy to check that under Hypothesis \ref{hypo:basic}, the family
$\{w^K_s:K\times S^1\to M\}_{s\in [s_0, \infty)}$ satisfies the following
\begin{enumerate}
\item $\delbar^\pi w^K_s=0$, $d((w^K_s)^*\lambda\circ j)=0$, for every $s\in [s_0, \infty)$
\item $\lim_{s\to \infty}\|d^\pi w^K_s\|_{L^2(K\times S^1)}=0$
\item $\|d w^K_s\|_{C^0(K\times S^1)}\leq \|d w\|_{C^0([0, \infty)\times S^1)}<\infty$.
\end{enumerate}

From (1) and (3) together with the compactness of the target manifold $M$ (which provides a uniform $L^2(K\times S^1)$ bound)
and Theorem \ref{thm:local-regularity}, we obtain
$$
\|w^K_s\|_{W^{3,2}(K\times S^1)}\leq C_{K;(3,2)}<\infty,
$$
for some constant $C_{K;(3,2)}$ independent of $s$.
Then by compactness of the embedding of $W^{3,2}(K\times S^1)$ into $C^{1, \alpha}(K\times S^1)$ for some $0< \alpha < 1$,
 $\{w^K_s:K\times S^1\to M\}_{s\in [s_0, \infty)}$ is sequentially pre-compact.
Therefore, for any sequence $s_k \to \infty$, there exists a subsequence, still denoted by $s_k$,
and some limit $w^K_\infty\in C^{1}(K\times S^1, M)$ (which may depend on the subsequence $\{s_k\}$), such that
$$
w^K_{s_k}\to w^K_{\infty}, \quad \text {as } k\to \infty,
$$
in the $C^{1}(K\times S^1, M)$-norm sense.
Further, combining this with (2), we get
$$
dw^K_{s_k}\to dw^K_{\infty} \quad \text{and} \quad dw^K_\infty=(w^K_\infty)^*\lambda\otimes X_\lambda,
$$
and both $(w^K_\infty)^*\lambda$ are $(w^K_\infty)^*\lambda\circ j$ are harmonic $1$-forms by (1).

Notice that these limiting maps $w^K_\infty$ have a common extension $w_\infty: \R\times S^1\to M$
by a  diagonal sequence argument where, one takes a sequence of compact sets $K$ one including another and exhausting $\R$.
Then $w_\infty$ is $C^1$,  satisfies
$$
\|d w_\infty\|_{C^0(\R\times S^1)}\leq \|d w\|_{C^0([0, \infty)\times S^1)}<\infty
$$
and $d^\pi w_\infty = 0$ and hence
$$
dw_\infty=(w_\infty)^*\lambda \otimes X_\lambda.
$$
Then we derive from Theorem \ref{thm:local-regularity} that $w_\infty$ is actually in $C^\infty$.
Also notice that both $(w_\infty)^*\lambda$ and $(w_\infty)^*\lambda\circ j$ are bounded harmonic $1$-forms on $\R\times S^1$,
and hence they must be written in  the form
$$
(w_\infty)^*\lambda=a\,d\tau+b\,dt, \quad (w_\infty)^*\lambda\circ j=b\,d\tau-a\,dt,
$$
where $a$, $b$ are some constants.
We will show that such $a$ and $b$ are actually related to $T$ and $Q$ as
$$
a=-Q, \quad b=T.
$$

By taking an arbitrary point $r\in K$, since $w_\infty|_{\{r\}\times S^1}$ is the limit of some sequence $w_{s_k}|_{\{r\}\times S^1}$
in the $C^1$ sense, we derive
\beastar
b& = & \int_{\{r\}\times S^1}(w_\infty|_{\{r\}\times S^1})^*\lambda
= \int_{\{r\}\times S^1}\lim_{k\to \infty}(w_{s_k}|_{\{r\}\times S^1})^*\lambda\\
&=&\lim_{k\to \infty}\int_{\{r\}\times S^1}(w_{s_k}|_{\{r\}\times S^1})^*\lambda
= \lim_{k\to \infty}\int_{\{r+s_k\}\times S^1}(w|_{\{r+s_k\}\times S^1})^*\lambda\\
&=&\lim_{k\to \infty}(T-\frac{1}{2}\int_{[r+s_k, \infty)\times S^1}|d^\pi w|^2)\\
&=&T-\lim_{k\to \infty}\frac{1}{2}\int_{[r+s_k, \infty)\times S^1}|d^\pi w|^2
= T;\\
\\
-a &= & \int_{\{r\}\times S^1}(w_\infty|_{\{r\}\times S^1})^*\lambda\circ j
= \int_{\{r\}\times S^1}\lim_{k\to \infty}(w_{s_k}|_{\{r\}\times S^1})^*\lambda\circ j\\
&=&\lim_{k\to \infty}\int_{\{r\}\times S^1}(w_{s_k}|_{\{r\}\times S^1})^*\lambda\circ j\\
&=&\lim_{k\to \infty}\int_{\{r+s_k\}\times S^1}(w|_{\{r+s_k\}\times S^1})^*\lambda\circ j
= Q.
\eeastar
Here in the derivation, we use Remark \ref{rem:TQ}.

As we have already seen in the proof of Proposition \ref{prop:abbas}, the image of $w_\infty$ is contained in
a single leaf of the Reeb foliation by the connectedness of $[0,\infty) \times S^1$.
Let $\gamma: \R \to M$ be a parametrisation of
the leaf so that $\dot \gamma = X_\lambda(\gamma)$. Then we can write $w_\infty(\tau, t)=\gamma(s(\tau, t))$, where
$s:\R\times S^1\to \R$ and $s=-Q\,\tau+T\,t+c_0$ since $ds=-Q\,d\tau+T\,dt$, where $c_0$ is some constant.

From this we derive that, if $T\neq 0$,
$\gamma$ is a closed orbit of period $T$.
If $T=0$ but $Q\neq 0$, we can only conclude that $\gamma$ is a Reeb trajectory parameterized by $\tau\in \R$.
Of course, if both $T$ and $Q$ vanish, $w_\infty$ is a constant map.

In summary, we have given the proof of the following subsequential convergence theorem.
This includes the special case of \cite[Theorem 31]{hofer} given in the framework of symplectization
which corresponds to the case $Q=0$, $T\neq 0$ and $K=\{0\}$ here.
Besides looking at two constants $T$ and $Q$,  this also strengthens
the convergence statement of \cite[Theorem 31]{hofer} in that the $s$-coordinates
do not enter into the convergence statement or its proof.
Moreover, uniform convergence on any compact subset $K\times S^1\subset [0, \infty)\times S^1$
(which enhances the result for $K=\{0\}$ shown in \cite{hofer}) is an important ingredient
which enables us to follow the three-interval method in deriving the exponential decay
result for the case of Morse--Bott type contact forms in \cite{oh-wang3} (see also Part II of \cite{oh-wang2}).

\begin{thm}[Subsequence Convergence]\label{thm:subsequence}
Let $w:[0, \infty)\times S^1\to M$ satisfy the contact instanton equations \eqref{eq:instanton}
and Hypothesis \eqref{eq:hypo-basic-pt}.

Then for any sequence $s_k\to \infty$, there exists a subsequence, still denoted by $s_k$, and a
massless instanton $w_\infty(\tau,t)$ (i.e., $E^\pi(w_\infty) = 0$)
on the cylinder $\R \times S^1$  such that
$$
\lim_{k\to \infty}w(s_k + \tau, t) = w_\infty(\tau,t)
$$
in the $C^l(K \times S^1, M)$ sense for any $l$, where $K\subset [0,\infty)$ is an arbitrary compact set.

Furthermore, $w_\infty$ has the formula $w_\infty(\tau,t)= \gamma(-Q\, \tau + T\, t)$, where $\gamma$ is some Reeb trajectory,
and for the case of $Q=0$ or $T\neq 0$, the trajectory $\gamma$ is a \emph{closed} Reeb orbit of $X_\lambda$ with period $T$.

\end{thm}

From the previous theorem, we  immediately get the following corollary.

\begin{cor}\label{cor:tangent-convergence}Let $w:[0, \infty)\times S^1\to M$ satisfy the contact instanton equations \eqref{eq:instanton}
and Hypothesis \eqref{eq:hypo-basic-pt}.
Then
\beastar
&&\lim_{s\to \infty}\left|\pi \frac{\del w}{\del\tau}(s+\tau, t)\right|=0, \quad
\lim_{s\to \infty}\left|\pi \frac{\del w}{\del t}(s+\tau, t)\right|=0\\
&&\lim_{s\to \infty}\lambda(\frac{\del w}{\del\tau})(s+\tau, t)=-Q, \quad
\lim_{s\to \infty}\lambda(\frac{\del w}{\del t})(s+\tau, t)=T
\eeastar
and
$$
\lim_{s\to \infty}|\nabla^l dw(s+\tau, t)|=0 \quad \text{for any}\quad l\geq 1.
$$
All the limits are uniform for $(\tau, t)$ in $K\times S^1$ with compact $K\subset \R$.

\end{cor}

\begin{proof} We first consider the first derivative estimate, i.e., the $C^1$-decay estimate.
If any of the above limits doesn't hold uniformly (take $|\pi \frac{\del w}{\del\tau}(s+\tau, t)|$ for example),
then there exists some $\epsilon_0>0$ and a sequence
$k\to \infty$, $(\tau_j, t_j)\in K\times S^1$ such that $|\pi \frac{\del w}{\del\tau}(s_k+\tau_j, t_j)|\geq \epsilon_0$.
Then we can take a subsequence limit $(\tau_j, t_j)\to (\tau_0, t_0)$ such that
$|\pi \frac{\del w}{\del\tau}(s_k+\tau_0, t_0)|\geq \frac{1}{2}\epsilon_0$ for $k$ large enough.

However, by Theorem \ref{thm:subsequence}, we can take a subsequence of $s_k$ such that
$w(s_k+\tau, t)$ converges to $\gamma(-Q\, \tau + T\, t)$ in a neighborhood of $(\tau_0, t_0)\in K\times S^1$,
in the $C^\infty$ sense. Here $\gamma$ is some Reeb trajectory.
Then we get $\lim_{s\to \infty}|\pi \frac{\del w}{\del\tau}(s_k+\tau_0, t_0)|=0$ and get a contradiction.

Once we establish this uniform $C^1$-decay result, the higher order decay result is
an immediate consequence of the uniform local pointwise higher order a priori estimates on the cylinder from
Theorem \ref{thm:local-higher-regularity}.
\end{proof}

\appendix

\section{The Weitzenb\"ock formula for vector valued forms}\label{appen:weitzenbock}

In this appendix, we recall the standard Weitzenb\"ock formulas applied to our
current circumstance. A good exposition on the general Weitzenb\"ock formula is
provided in the appendix of \cite{freed-uhlen}.

Assume $(P, h)$ is a Riemannian manifold of dimension $n$ with metric $h$, and $D$ is the Levi-Civita connection.
Let $E\to P$ be any vector bundle with inner product $\langle\cdot, \cdot\rangle$,
and assume $\nabla$ is a connection on $E$ which is compatible with $\langle\cdot, \cdot\rangle$.

For any $E$-valued form $s$, calculating the (Hodge) Laplacian of the energy density
of $s$,  we get
\beastar
-\frac{1}{2}\Delta|s|^2=|\nabla s|^2+\langle Tr\nabla^2 s, s\rangle,
\eeastar
where for $|\nabla s|$ we mean the induced norm in the vector bundle $T^*P\otimes E$, i.e.,
$|\nabla s|^2=\sum_i|\nabla_{E_i}s|^2$ with $\{E_i\}$ an orthonormal frame of $TP$.
$Tr\nabla^2$ denotes the connection Laplacian, which is defined as
$Tr\nabla^2=\sum_i\nabla^2_{E_i, E_i}s$,
where $\nabla^2_{X, Y}:=\nabla_X\nabla_Y-\nabla_{\nabla_XY}$.

Denote by $\Omega^k(E)$ the space of $E$-valued $k$-forms on $P$. The connection $\nabla$
induces an exterior derivative by
\beastar
d^\nabla&:& \Omega^k(E)\to \Omega^{k+1}(E)\\
d^\nabla(\alpha\otimes \zeta)&=&d\alpha\otimes \zeta+(-1)^k\alpha\wedge \nabla\zeta.
\eeastar

It is not hard to check that for any $1$-forms, equivalently one can write
$$
d^\nabla\beta (v_1, v_2)=(\nabla_{v_1}\beta)(v_2)-(\nabla_{v_2}\beta)(v_1),
$$
where $v_1, v_2\in TP$.

We extend the Hodge star operator to $E$-valued forms by
\beastar
*&:&\Omega^k(E)\to \Omega^{n-k}(E)\\
*\beta&=&*(\alpha\otimes\zeta)=(*\alpha)\otimes\zeta
\eeastar
for $\beta=\alpha\otimes\zeta\in \Omega^k(E)$.

Define the Hodge Laplacian of the connection $\nabla$ by
$$
\Delta^{\nabla}:=d^{\nabla}\delta^{\nabla}+\delta^{\nabla}d^{\nabla},
$$
where $\delta^{\nabla}$ is defined by
$$
\delta^{\nabla}:=(-1)^{nk+n+1}*d^{\nabla}*.
$$
The following lemma is important for the derivation of the Weitzenb\"ock formula.
\begin{lem}\label{lem:d-delta}Assume $\{e_i\}$ is an orthonormal frame of $P$, and $\{\alpha^i\}$ is the dual frame.
Then we have
\beastar
d^{\nabla}&=&\sum_i\alpha^i\wedge \nabla_{e_i}\\
\delta^{\nabla}&=&-\sum_ie_i\rfloor \nabla_{e_i}.
\eeastar
\end{lem}
\begin{proof}Assume $\beta=\alpha\otimes \zeta\in \Omega^k(E)$. Then
\beastar
d^{\nabla}(\alpha\otimes \zeta)&=&(d\alpha)\otimes \zeta+(-1)^k\alpha\wedge\nabla\zeta\\
&=&\sum_i\alpha^i\wedge \nabla_{e_i}\alpha\otimes\zeta+(-1)^k\alpha\wedge\nabla\zeta.
\eeastar
On the other hand,
\beastar
\sum_i\alpha^i\wedge \nabla_{e_i}(\alpha\otimes\zeta)&=&
\sum_i\alpha^i\wedge\nabla_{e_i}\alpha\otimes\zeta+\alpha^i\wedge\alpha\otimes\nabla_{e_i}\zeta\\
&=&\sum_i\alpha^i\wedge \nabla_{e_i}\alpha\otimes\zeta+(-1)^k\alpha\wedge\nabla\zeta,
\eeastar
so we have proved the first statement.

For the second equality, we compute
\beastar
\delta^{\nabla}(\alpha\otimes\zeta)&=&(-1)^{nk+n+1}*d^{\nabla}*(\alpha\otimes\zeta)\\
&=&(\delta\alpha)\otimes\zeta+(-1)^{nk+n+1}*(-1)^{n-k}(*\alpha)\wedge\nabla\zeta\\
&=&-\sum_ie_i\rfloor \nabla_{e_i}\alpha\otimes\zeta+\sum_i(-1)^{nk-k+1}*((*\alpha)\wedge\alpha^i)\otimes\nabla_{e_i}\zeta\\
&=&-\sum_ie_i\rfloor \nabla_{e_i}\alpha\otimes\zeta-\sum_ie_i\rfloor \alpha\otimes\nabla_{e_i}\zeta\\
&=&-\sum_ie_i\rfloor \nabla_{e_i}(\alpha\otimes\zeta).
\eeastar
\end{proof}

\begin{thm}[Weitzenb\"ock Formula]\label{thm:weitzenbock}
Assume $\{e_i\}$ is an orthonormal frame of $P$, and $\{\alpha^i\}$ is the dual frame. Then
when applied to $E$-valued forms
\beastar
\Delta^{\nabla}=-Tr\nabla^2+\sum_{i,j}\alpha^j\wedge (e_i\rfloor R(e_i,e_j)\cdot)
\eeastar
where $R$ is the curvature tensor of the bundle $E$ with respect to the connection $\nabla$.
\end{thm}
\begin{proof}Since the right hand side of the equality is independent of the choice of orthonormal basis,
and it is a pointwise formula,
we can take the normal coordinates $\{e_i\}$ at a point $p\in P$ (and $\{\alpha^i\}$ the dual basis), i.e., $h_{ij}:=h(e_i, e_j)(p)=\delta_{ij}$ and $dh_{i,j}(p)=0$,
 and prove that the given formula holds at $p$ for such coordinates. For the Levi-Civita connection, the condition $dh_{i,j}(p)=0$
 of the normal coordinate is equivalent to letting
$\Gamma^k_{i,j}(p):=\alpha^k(D_{e_i}e_j)(p)$ be $0$.

For $\beta\in \Omega^k(E)$, using Lemma \ref{lem:d-delta} we calculate
\beastar
\delta^{\nabla}d^{\nabla}\beta&=&-\sum_{i,j}e_i\rfloor \nabla_{e_i}(\alpha^j\wedge\nabla_{e_j}\beta)\\
&=&-\sum_{i,j}e_i\rfloor (D_{e_i}\alpha^j \wedge\nabla_{e_j}\beta+\alpha^j\wedge \nabla_{e_i}\nabla_{e_j}\beta).
\eeastar
At the point $p$, the first term vanishes, and we get
\beastar
\delta^{\nabla}d^{\nabla}\beta(p)&=&-\sum_{i,j}e_i\rfloor (\alpha^j\wedge \nabla_{e_i}\nabla_{e_j}\beta)(p)\\
&=&-\sum_i\nabla_{e_i}\nabla_{e_i}\beta(p)+\sum_{i,j}\alpha^j\wedge (e_i\rfloor\nabla_{e_i}\nabla_{e_j}\beta)(p)\\
&=&-\sum_i\nabla^2_{e_i, e_i}\beta(p)+\sum_{i,j}\alpha^j\wedge (e_i\rfloor\nabla_{e_i}\nabla_{e_j}\beta)(p).
\eeastar
Also,
\beastar
d^{\nabla}\delta^{\nabla}\beta&=&-\sum_{i,j}\alpha^i\wedge \nabla_{e_i}(e_j\rfloor \nabla_{e_j}\beta)\\
&=&-\sum_{i,j}\alpha^i\wedge (e_j\rfloor \nabla_{e_i}\nabla_{e_j}\beta)
-\sum_{i,j}\alpha^i\wedge ((D_{e_i}e_j) \rfloor \nabla_{e_j}\beta).
\eeastar
As before, at the point $p$, the second term vanishes.

Now we sum the two parts $d^\nabla\delta^\nabla$ and $\delta^\nabla d^\nabla$ and get
$$
\Delta^{\nabla}\beta(p)=-\sum_i\nabla^2_{e_i, e_i}\beta(p)
+\sum_{i,j}\alpha^j\wedge (e_i\rfloor R(e_i,e_j)\beta)(p).
$$
\end{proof}

In particular, when acting on zero forms, i.e., sections of $E$, the second term on the right hand side vanishes, and there is
$$
\Delta^{\nabla}=-Tr\nabla^2.
$$
When acting on full rank forms, the above also holds by easy checking.

When $\beta\in \Omega^1(E)$, which is the case we use in this article, there is the following
\begin{cor}
For $\beta=\alpha\otimes \zeta\in \Omega^1(E)$, the Weizenb\"ock formula can be written as
$$
\Delta^{\nabla}\beta=-\sum_i\nabla^2_{e_i, e_i}\beta
+\ric^{D*}(\alpha)\otimes\zeta+\ric^{\nabla}\beta,
$$
where $\ric^{D*}$ denotes the adjoint of $\ric^D$, which acts on $1$-forms.

In particular, when $P$ is a surface, we have
\bea\label{eq:bochner-weitzenbock}
\Delta^{\nabla}\beta&=&-\sum_i\nabla^2_{e_i, e_i}\beta
+K\cdot\beta+\ric^{\nabla}(\beta)\nonumber\\
-\frac{1}{2}\Delta|\beta|^2&=&|\nabla \beta|^2-\langle \Delta^\nabla \beta, \beta\rangle +K\cdot|\beta|^2+\langle\ric^{\nabla}(\beta), \beta\rangle,
\eea
where $K$ is the Gaussian curvature of the surface $P$, and $\ric^{\nabla}(\beta):=\alpha\otimes \Sigma_{i, j}R(e_i, e_j)\zeta$.
\end{cor}

\section{Wedge products of vector-valued forms}
\label{appen:forms}

In this section, we continue with the setting from Appendix \ref{appen:weitzenbock}.
To be specific, we assume $(P, h)$ is a Riemannian manifold of dimension $n$ with metric $h$, and denote by $D$
the Levi-Civita connection. $E\to P$ is a vector bundle with inner product $\langle \cdot, \cdot\rangle$
and $\nabla$ is a connection of $E$ which is compatible with $\langle \cdot, \cdot\rangle$.

We remark that we include this section for the sake of completeness of our treatment of vector valued forms, and
the content of this appendix is not used in any section of this article.
Actually one can derive exponential decay using the differential inequality method from the formulas we provide here.
We leave the proof to interested reader.

The wedge product of forms can be extended to $E$-valued forms by defining
\beastar
\wedge&:&\Omega^{k_1}(E)\times \Omega^{k_2}(E)\to \Omega^{k_1+k_2}(P)\\
\beta_1\wedge\beta_2&=&\langle \zeta_1, \zeta_2\rangle\,\alpha_1\wedge\alpha_2,
\eeastar
where $\beta_1=\alpha_1\otimes\zeta_1\in \Omega^{k_1}(E)$ and $\beta_2=\alpha_2\otimes\zeta_2\in \Omega^{k_2}(E)$
are $E$-valued forms.

\begin{lem}\label{lem:inner-star}
For $\beta_1, \beta_2\in \Omega^k(E)$,
$$
\langle \beta_1, \beta_2\rangle=*(\beta_1\wedge *\beta_2).
$$
\end{lem}
\begin{proof}
Write $\beta_1=\alpha_1\otimes\zeta_1$ and $\beta_2=\alpha_2\otimes\zeta_2$. Then
\beastar
*(\beta_1\wedge *\beta_2)&=&*\big((\alpha_1\otimes\zeta_1)\wedge ((*\alpha_2)\otimes\zeta_2)\big)\\
&=&*(\langle\zeta_1, \zeta_2\rangle\,\alpha_1\wedge *\alpha_2)\\
&=&\langle\zeta_1, \zeta_2\rangle\,*(\alpha_1\wedge *\alpha_2)\\
&=&\langle\zeta_1, \zeta_2\rangle\, h(\alpha_1, \alpha_2)\\
&=&\langle \beta_1, \beta_2\rangle.
\eeastar
\end{proof}

The following lemmas exploit the compatibility of $\nabla$ with the inner product $\langle \cdot, \cdot\rangle$.

\begin{lem}
$$
d(\beta_1\wedge\beta_2)=d^\nabla\beta_1\wedge \beta_2+(-1)^{k_1}\beta_1\wedge d^\nabla\beta_2,
$$
where $\beta_1\in \Omega^{k_1}(E)$ and $\beta_2\in \Omega^{k_2}(E)$
are $E$-valued forms.
\end{lem}

\begin{proof}
We write $\beta_1=\alpha_1\otimes\zeta_1$ and $\beta_2=\alpha_2\otimes\zeta_2$ and calculate
\beastar
d(\beta_1\wedge\beta_2)&=&d(\langle \zeta_1, \zeta_2\rangle\,\alpha_1\wedge\alpha_2)\\
&=&d\langle \zeta_1, \zeta_2\rangle\wedge\alpha_1\wedge\alpha_2+\langle \zeta_1, \zeta_2\rangle\,d(\alpha_1\wedge\alpha_2)\\
&=&\langle \nabla\zeta_1, \zeta_2\rangle\wedge\alpha_1\wedge\alpha_2+\langle \zeta_1, \nabla\zeta_2\rangle\wedge\alpha_1\wedge\alpha_2\\
&{}&+\langle \zeta_1, \zeta_2\rangle\,d\alpha_1\wedge\alpha_2+(-1)^{k_1}\langle \zeta_1, \zeta_2\rangle\,\alpha_1\wedge d\alpha_2,
\eeastar
while
\beastar
d^\nabla\beta_1\wedge \beta_2&=&d^\nabla(\alpha_1\otimes\zeta_1)\wedge(\alpha_2\otimes\zeta_2)\\
&=&(d\alpha_1\otimes\zeta_1+(-1)^{k_1}\alpha_1\wedge\nabla\zeta_1)\wedge(\alpha_2\otimes\zeta_2)\\
&=&\langle\zeta_1, \zeta_2\rangle\,d\alpha_1\wedge\alpha_2+\langle\nabla\zeta_1, \zeta_2\rangle\wedge\alpha_1\wedge\alpha_2.
\eeastar
A similar calculation shows that
$$
(-1)^{k_1}\beta_1\wedge d^\nabla\beta_2=
(-1)^{k_1}\langle \zeta_1, \zeta_2\rangle\,\alpha_1\wedge d\alpha_2+\langle \zeta_1, \nabla\zeta_2\rangle\wedge\alpha_1\wedge\alpha_2.
$$
Summing these up,  we get the equality we want.
\end{proof}

\begin{lem}\label{lem:metric-property}
Assume $\beta_0\in \Omega^k(E)$ and $\beta_1\in \Omega^{k+1}(E)$, then we have
$$
\langle d^\nabla \beta_0, \beta_1\rangle-(-1)^{n(k+1)}\langle \beta_0, \delta^\nabla\beta_1\rangle
=*d(\beta_0\wedge *\beta_1).
$$
\end{lem}
\begin{proof}
We calculate
\beastar
*d(\beta_0\wedge *\beta_1)&=&*\big(d^\nabla\beta_0\wedge*\beta_1+(-1)^k\beta_0\wedge (d^\nabla*\beta_1)\big)\\
&=&\langle d^\nabla\beta_0, \beta_1\rangle
+(-1)^n*\big(\beta_0\wedge *(*d^\nabla *\beta_1\big)\\
&=&\langle d^\nabla\beta_0, \beta_1\rangle
-(-1)^{n(k+1)}\langle\beta_0, \delta^\nabla\beta_1\rangle.
\eeastar

\end{proof}

\section{Local coercive estimates}
\label{appen:local-coercive}

In this appendix, we
give the proof of Proposition \ref{prop:coercive-L2} which we restate here.

\begin{prop}\label{prop:coercive-L2app}
For any open domains $D_1$ and $D_2$ in $\dot\Sigma$ satisfying $\overline{D_1}\subset D_2$,
$$
\|\nabla(dw)\|^2_{L^2(D_1)}\leq C_1(D_1, D_2)\|dw\|^2_{L^2(D_2)}+C_2(D_1, D_2)\|dw\|^4_{L^4(D_2)}
$$
for any contact instanton $w$,
where $C_1(D_1, D_2)$ and $C_2(D_1, D_2)$ are some constants, which are independent of $w$.
\end{prop}
\begin{proof}
For the pair of given domains $D_1$ and $D_2$, we choose another domain $D$ such that
$\overline D_1 \subset D \subset \overline D \subset D_2$ and a smooth cut-off function $\chi:D_2\to \R$ such that
$\chi\geq 0$ and
$\chi\equiv 1$ on $\overline{D_1}$, $\chi\equiv 0$ on $D_2-D$.
Multiplying  \eqref{eq:higher-derivative} by $\chi^2$ and integrating over $D$, we get
\beastar
\int_{D_1}|\nabla(dw)|^2&\leq&\int_{D}\chi^2|\nabla(dw)|^2\\
&\leq&C_1\int_{D}\chi^2|dw|^4-4\int_{D}K\chi^2|dw|^2-2\int_{D}\chi^2\Delta e\\
&\leq&C_1\int_{D_2}|dw|^4+4\|K\|_{L^\infty(\dot\Sigma)}\int_{D_2}|dw|^2-2\int_{D}\chi^2\Delta e
\eeastar
where $C_1$ is the same constant as the one appearing in \eqref{eq:higher-derivative}.

We now deal with the last term $\int_{D_2}\chi^2 \Delta e$.
Since
\beastar
\chi^2\Delta e\, dA&=&*(\chi^2 \Delta e)=\chi^2 *\Delta e
=-\chi^2 d*de\\
&=&-d(\chi^2 *de)+2\chi d\chi\wedge (*de),
\eeastar
we get
$$
\int_{D}\chi^2\Delta e\, dA=\int_{D}2\chi d\chi\wedge (*de)
$$
by integrating the identity over $D$ and applying Stokes' formula. Here we use the fact that $\chi$ vanishes on $D_2-D$, in particular
on $\del D$.

To deal with the right hand side, we have
\beastar
&{}&|\int_{D}\chi d\chi\wedge(*de)|
=|\int_{D}\chi\langle d\chi, de\rangle \,dA|
\leq\int_{D}|\chi||\langle d\chi, de\rangle \,dA|\leq\int_{D}|\chi||d\chi||de|\,dA.
\eeastar
Notice also
\beastar
|de|&=&|d\langle dw, dw\rangle|=2|\langle \nabla(dw), dw\rangle|\leq 2|\nabla (dw)||dw|.
\eeastar
Hence
\beastar
|\int_{D}\chi d\chi\wedge(*de)|&\leq& \int_{D}2|\chi||d\chi||\nabla (dw)||dw|\,dA\\
&\leq&  \frac{1}{\epsilon}\int_{D}\chi^2|\nabla(dw)|^2\,dA+\epsilon\int_{D}|d\chi|^2|dw|^2\,dA\\
&\leq&  \frac{1}{\epsilon}\int_{D}\chi^2|\nabla(dw)|^2\,dA+\epsilon\|d\chi\|_{C^0(D)}^2\int_{D}|dw|^2\,dA
\eeastar
Then we can sum all the estimates above and get
\beastar
\int_{D}\chi^2|\nabla(dw)|^2
&\leq& \int_D\frac{2\chi^2}{\epsilon}|\nabla(dw)|^2\\
&{}&+\left(4\|K\|_{L^\infty(\dot\Sigma)}+2\|d\chi\|_{C^0(D)}\epsilon\right)\int_{D_2}|dw|^2\\
&{}&+C_1\int_{D_2}|dw|^4.
\eeastar
We take $\epsilon=4$. Then
\beastar
&{}&\int_{D_1}|\nabla(dw)|^2\leq\int_{D}\chi^2|\nabla(dw)|^2\nonumber\\
&\leq&\left(8\|K\|_{L^\infty(\dot\Sigma)}+16\|d\chi\|^2_{C^0(D)}\right)\int_{D_2}|dw|^2
+2C_1\int_{D_2}|dw|^4.\nonumber\\
\eeastar
By setting $C_1(D_1,D_2) = 8\|K\|_{L^\infty(\dot\Sigma)}+16\|d\chi\|^2_{C^0(D)}$ and $C_2(D_1,D_2) = 2C_1$
with $C_1$ the constant given in \eqref{eq:higher-derivative},
we have finished the proof.
\end{proof}

\bigskip

\textbf{Acknowledgements:} For their valuable feedback, we thank the audiences of our talks on
this topic in the seminars of various institutions.
Rui Wang sincerely thanks Bohui Chen for
numerous mathematical discussions as well as for his continuous encouragement.
Both authors greatly thank Gabriel C. Drummond-Cole for his help in improving the English expression of this article.
We also thank the anonymous referee for her/his careful reading of the paper and pointing out errors and many typos
which much improves the presentation of this paper.

\end{document}